\documentclass[12pt]{article}
\usepackage{amsfonts}
\usepackage{amsmath}
\usepackage{amssymb}
\usepackage[mathscr]{eucal}
\usepackage{graphicx}
\usepackage{soul,xcolor}
\usepackage{fullpage}
\usepackage{times}
\usepackage{bbm}
\usepackage{epsfig}
\usepackage{color}
\usepackage[pagebackref]{hyperref}    
\usepackage{caption}
\usepackage{mathtools}

\newcommand{\overbar}[1]{\mkern 1.5mu\overline{\mkern-1.5mu#1\mkern-1.5mu}\mkern 1.5mu}

\def\eod{\vrule height 6pt width 5pt depth 0pt}
\newenvironment{proof}{\noindent {\bf Proof:} \hspace{.2em}}
                      {\hspace*{\fill}{\eod}}

\newtheorem{theorem}{Theorem}
\newtheorem{lemma}[theorem]{Lemma}
\newtheorem{definition}[theorem]{Definition}
\newtheorem{corollary}[theorem]{Corollary}

\newtheorem{example}[theorem]{Example}
\newtheorem{remark}[theorem]{Remark}

\newcommand{\ED}{ \mathsf{ED}}

\newcommand{\sD}{ \mathcal{ D}}
\newcommand{\sL}{\mathcal{ L}}

\newcommand{\RR}{ \mathbbm{R}}

\newcommand{\comment}[1]{}

\newcommand{\SSS}{\mathfrak{S}}

\newcommand{\CC}{ \mathbb{C}}

\newcommand{\GTS}{\mathsf{GTS}}

\newcommand{\Tr}{ \mathsf{Trace}}

\newcommand{\obr}[1]{\overbar{#1}}

\newcommand{\sF}{ \mathcal{ F}}

\begin{document}

\title{Eigenvalue monotonicity of $q$-Laplacians of trees along a poset} 

\author{Mukesh Kumar Nagar\\
Department of Mathematics\\
Indian Institute of Technology, Bombay\\
Mumbai 400 076, India.\\
email: mukesh.kr.nagar@gmail.com
}

\maketitle

\begin{abstract}
Let $T$ be a tree on $n$ vertices with $q$-Laplacian $\sL_{T}^{q}$.
Let $\GTS_n$ be the generalized 
tree shift poset on the set of unlabelled trees with $n$ vertices. 
We prove that for all $q \in \RR$, going up on $\GTS_n$ has the 
following effect: 
the spectral radius  and the second smallest eigenvalue of 
$\sL_{T}^{q}$ increase while the smallest eigenvalue of $\sL_{T}^{q}$ 
decreases. These generalize known results for eigenvalues of the 
Laplacian.  As a corollary, we obtain consequences about the eigenvalues of
$q,t$-Laplacians and exponential distance matrices of trees.
\end{abstract}




\section{Introduction} 

In \cite{kelmans}, 
Kelmans defined an operation on graphs 
 called the Kelmans' transformation. 
This transformation has nice properties: 
it increases the spectral radius of adjacency matrix (see  Csikv{\'a}ri 
\cite{csikvari-on-conjecture})
and decreases  the number of spanning trees 
(see Satyanarayana,  Schoppman and Suffel 
\cite{satyanarayana-schoppman-suffel}). 

Motivated by Kelmans' transformation, 
Csikv{\'a}ri \cite{csikvari-poset1,csikvari-poset2} defined a  poset  
on the set of unlabelled trees with $n$ vertices denoted 
here as $\GTS_n$ (see Subsection \ref{subsec:GTS} for the definition). 
Among other results, he proved the following.

\begin{theorem}[Csikv{\'a}ri]
\label{thm:csikvari_main}
Going up on $\GTS_n$ increases both the  algebraic connectivity  
(the second smallest eigenvalue) and 
the largest eigenvalue of the Laplacian matrix of trees.
\end{theorem}

For a graph $G$, its Laplacian matrix $L_G$ has a $q$-analogue, 
denoted by $\sL_{G}^q$ called the $q$-Laplacian of $G$.  
The entries of $\sL_{G}^q$ are 
polynomials in a real variable $q$.  
The matrix  $\sL_G^q$ is defined as $\sL_G^q=I+q^2(D-I)-qA$, 
where $D$ is the diagonal matrix with degrees on the diagonal 
and $A$ is the adjacency matrix of $G$.   
Clearly when $q=1$,  $\sL_G^q = L_G$.  
The matrix $\sL_G^q$ has connections with 
the Ihara-Selberg zeta function of a graph
(see Bass \cite{bass} and Foata and Zeilberger \cite{foata-zeilberger-bass-trams}).  
When $G$ is a tree $T$,  $\sL_T^q$ has connections with distance matrix  (see 
Bapat, Lal and Pati \cite{bapat-lal-pati}).

In this paper, 
we prove the following more general result about the eigenvalues of  $\sL_T^q$.   
For a fixed  $q\in \RR$, 
let $\lambda_{\max}(\sL_T^{q})$,  $\lambda_{\min}(\sL_{T}^{q})$ and  $\lambda_{a}(\sL_{T}^{q})$ 
be the largest, the smallest and the second smallest eigenvalues of $\sL_{T}^{q}$ respectively.  

\begin{theorem}
\label{thm:main_result}
Let $T_1$ and $T_2$ be two trees with $n$ vertices such that 
$T_1 \leq_{\GTS_n} T_2$ on $\GTS_n$. Let $\sL_{T_1}^{q}$ and 
$\sL_{T_2}^{q}$ be the $q$-Laplacians of $T_1$ and $T_2$ respectively. 
Then,  for all $q\in \RR$, we have    

$$\lambda_{\max}(\sL_{T_1}^{q}) \leq \lambda_{\max}(\sL_{T_2}^{q}),
 \lambda_{a}(\sL_{T_1}^{q}) 
 \leq \lambda_{a}(\sL_{T_2}^{q})  \mbox{ and }\lambda_{\min}(\sL_{T_1}^{q}) 
\geq \lambda_{\min}(\sL_{T_2}^{q}).$$ 
\end{theorem} 

Thus for all $q \in \RR$, three eigenvalues of $\sL_T^q$ exhibit monotonicity when 
we go up on $\GTS_n$. 
It is easy to see that 
Theorem \ref{thm:main_result} 
gives us Theorem \ref{thm:csikvari_main} by setting
$q=1$.
Theorem \ref{thm:main_result} gives us one extra result which
is trivial when $q=1$, as it is well known that the smallest 
eigenvalue of the Laplacian $L_G$ of a graph $G$ is zero. 
Thus it is constant on $\GTS_n$.  

Note that when $q=0$,  $\sL_{T}^{q}$ is the identity matrix for all trees $T$. 
In this case, all the eigenvalues of $\sL_{T}^{q}$ 
on $\GTS_n$ are constant and  
hence Theorem \ref{thm:main_result} is trivially  true. 
Thus in this work, we will assume $q\neq 0$. 

The Laplacian  $L_G$  has a bivariate analogue 
denoted by $\sL_{G}^{q,t}$ called the $q,t$-Laplacian of $G$ 
(see Section \ref{sec:ev_L^qt} for the definition). 
 $\sL_G^{q,t}$ was defined by 
Bapat and Sivasubramanian in \cite{bapat-siva-prod_dist_matrix} 
to get a bivariate version of the  Ihara-Selberg zeta function of $G$.
When the graph $G$ is a tree $T$, 
$\sL_{T}^{q,t}$ has connections with bivariate versions of distance matrices 
(see \cite{bapat-siva-prod_dist_matrix}). 
Here, both $q$ and $t$ are variables 
and we will let them take both real and complex values. 
Our results have implication for the eigenvalues of 
$\sL_{T}^{q,t}$ for some values of $q,t\in \CC$.

This paper is organized as follows:  
Section \ref{sec:prelim} describes some preliminaries on the poset $\GTS_n$ and exponential distance matrix $\ED_T$ of a tree $T$.  
In Section \ref{sec:gen_lemma}, 
we extend the general lemma proved by Csikv{\'a}ri \cite{csikvari-poset2} 
to the characteristic polynomial of $\sL_T^q$. 
In Section \ref{sec:interlacing}, we prove an interlacing  
results about the eigenvalues of $\sL_T^q$ for all $q \in \RR$. 
In Section \ref{sec:F_T}, inspired by Csikv{\'a}ri \cite{csikvari-poset2}, 
we define and prove some 
properties of an auxiliary bivariate  
polynomial which  in Section \ref{sec:eigenvalue_q_Lap} is used to prove 
Theorem \ref{thm:main_result}.
Theorem \ref{thm:main_result} 
can be used  to 
obtain upper bounds on the largest and 
the second smallest  eigenvalues of $\sL_T^q$ for all $q\in \RR$, 
see Corollaries \ref{cor:max_ev_bound_L_Tq} and \ref{cor:2nd_smal_ev_bound} respectively.
Theorem \ref{thm:main_result} also has consequences for eigenvalues of the $q,t$-Laplacian
and the exponential distance matrices of trees.
These are  obtained in Sections 
\ref{sec:ev_L^qt} and \ref{sec:expon_dist_mat} respectively.

\section{Preliminaries}
\label{sec:prelim}

In the first part we give some preliminaries on the poset $\GTS_n$ 
and in the second subsection we cover some preliminary results of the 
$q$-Laplacian matrix $\sL_T^q$ and of the  
exponential distance matrix $\ED_T$ of a tree $T$.

\subsection{The Poset $\GTS_n$}
\label{subsec:GTS}

We recall the definition of the generalized tree shift poset $\GTS_n$ 
defined by Csikv{\'a}ri \cite{csikvari-poset1}. 
Let  $P_n$ and $S_n$ denote  
the path tree and the star tree on $n$ vertices respectively.  

\begin{definition}
\label{def:GTS}
Let $T_1$ be a tree on $n$ vertices. 
Let   $P_{k}$ be a path between two vertices in $T_1$, 
say  $1$ and $k$  such that 
each internal vertex (if it exists) of $P_{k}$ has degree 2. 
Let $k-1$ be the neighbour of $k$ on $P_{k}$. 
Construct a new tree $T_2$ from $T_1$ 
by moving all the neighbours of  $k$ 
other than $k-1$ to the vertex $1$. 
This operation is called a generalized tree shift. 
This is illustrated in Figure \ref{fig:gts_example_lastchap}. 
\end{definition}

\begin{figure}[h]
\centerline{\includegraphics[scale=0.6]{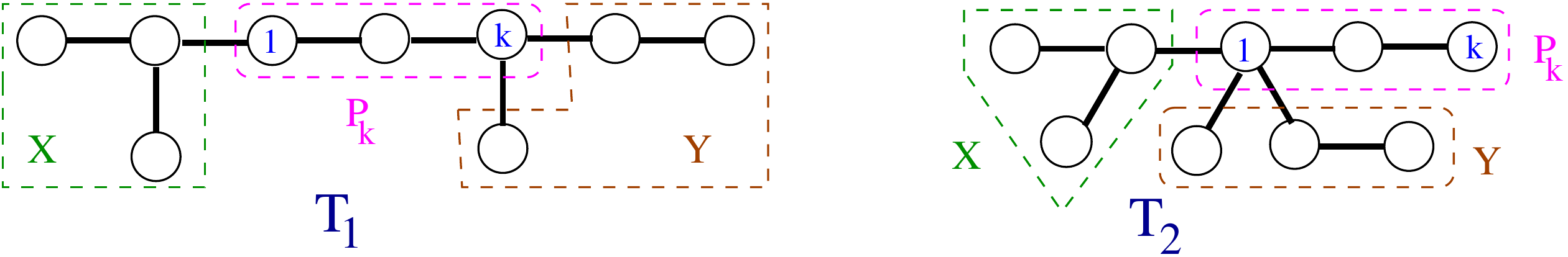}}
\caption{Example of $T_1 \leq_{\GTS_n} T_2$.}
\label{fig:gts_example_lastchap}
\end{figure}
The generalized tree shift operation gives us a partial order  
on the set of unlabelled trees with $n$ vertices, denoted by ``$\leq_{\GTS_n}$''.
If $T_1\leq_{\GTS_n} T_2$ and there is no tree $T$ with $T\neq T_1, T_2$ 
such that $T_1 \leq_{\GTS_n} T \leq_{\GTS_n} T_2$, 
then we say $T_2$ covers $T_1$ in $\GTS_n$. 
If either $1$ or $k$ is a leaf vertex in $T_1$  
then it is simple to check that $T_2$ is isomorphic to $T_1$. 
If neither vertex $1$ nor vertex $k$ is a leaf in $T_1$  
then $T_2$ covers $T_1$. In this case,  the number of leaf vertices in     
$T_2$ is one more than the number of leaf vertices in $T_1$. 
Conversely, if $T_2$ covers $T_1$ then there exists vertices $1,k$ 
which witnesses the covering relation. 
We will use the vertices $1,k$ only for this purpose in this paper. 
We refer the reader to Csikv{\'a}ri \cite{csikvari-poset2} for Hasse diagram of $\GTS_6$.  
Csikv{\'a}ri \cite[Theorem 2.4 and Corollary 2.5]{csikvari-poset1} proved the following result.
  
\begin{lemma}[Csikv{\'a}ri]
\label{lem:csikvari_min_max}
Let $T$ be a tree with $n$ vertices different from $P_n$. 
Then, there exists $T'$ such that    $T' \leq_{\GTS_n} T$. 
Let $T$ be a tree with $n$ vertices different from $S_n$. 
Then, there exists $T''$ such that    $T \leq_{\GTS_n} T''$.
Moreover, $P_n$ and $S_n$ are the only  
minimal and the maximal elements of  $\GTS_n$ respectively. 
\end{lemma}

Thus monotonicity results on $\GTS_n$ show that max-min pair is either $(P_n,S_n)$ or $(S_n,P_n)$.

\subsection{$q$-Laplacian and Exponential distance matrix of a tree}
\label{subsec:ex-dist-mat}

For a graph $G$, let $B_1=\sL_G^q|_{q=1}^{}=L_G=D-A$ and $B_2=\sL_G^q|_{q=-1}^{}=D+A$, where $D$ is the diagonal matrix with degrees on the diagonal 
and $A$ is the adjacency matrix of $G$. 
It is well known that  $B_1$ and $B_2$ are   
similar matrices for a bipartite graph $G$. 
Thus, when $q=\pm 1$, the $q$-Laplacian matrix $\sL_T^q$ of a tree $T$ is positive semidefinite. In this case, all the eigenvalues of $\sL_{T}^{q}$ are non-negative and 
the multiplicity of zero as an eigenvalue of $\sL_{T}^{q}$ is $1$. 
Bapat, Lal and Pati  \cite[Propositions 3.4 and 3.7]{bapat-lal-pati} proved the following result.

\begin{lemma}[Bapat, Lal and Pati] 
	\label{lem:bapat-lal-pati}
	Let $T$ be a tree on $n$ vertices with $q$-Laplacian $\sL_{T}^{q}$. Then, 
	\begin{enumerate}
		\item $\det (\sL_T^q)=1-q^2$.
		\item $\sL_{T}^{q}$ is positive definite when  $q\in \RR$ with $|q|<1$. 
		\item $\sL_{T}^{q}$ has exactly one negative eigenvalue when $q\in \RR$ with $|q|>1$.
	\end{enumerate}
\end{lemma}

In \cite{bapat-lal-pati}, Bapat, Lal and Pati introduced the 
exponential distance matrix $\ED_T$ of a tree $T$. 
We recall its definition, let $T$ be a tree with $n$ vertices, 
then its exponential distance matrix   
$\ED_T=(e_{i,j})_{1\leq i,j\leq n}$ is defined as follows:  
the entry $e_{i,j}=1$, if $i=j$ and $e_{i,j}=q^{d_{i,j}}$, if $i\neq j$, 
where $d_{i,j}$ is the distance between vertex $i$ and vertex $j$ in $T$. 
Clearly $\ED_T$ is a symmetric matrix, hence all its eigenvalues are real.   
Bapat, Lal and Pati \cite[Lemma 3.8]{bapat-lal-pati} proved the following result 
about the relationship between the eigenvalues of $\sL_{T}^{q}$ and $\ED_T$. 
\begin{lemma}[Bapat, Lal and Pati]
	\label{lem:bapat_ED_ev}
	Let $T$ be a tree with $n$ vertices. 
	Let $\sL_{T}^{q}$ and $\ED_T$ be the  $q$-Laplacian 
	and the exponential distance matrix of $T$ respectively. 
	If $q\neq \pm 1$, then $\ED_T^{-1}=\frac{1}{1-q^2} \sL_{T}^{q}.$
	Moreover,  $\frac{1-q^2}{\lambda_i(\sL_{T}^{q})}$ is an eigenvalue of $\ED_T$, 
	where $\lambda_i(\sL_{T}^{q})$ is an eigenvalue of $\sL_{T}^{q}$. 
\end{lemma}

Let $T$ be a tree on $n$ vertices with $q$-Laplacian $\sL_{T}^{q}$ and exponential distance matrix $\ED_T$. 
Let the eigenvalues of $\sL_T^q$ be  
$\lambda_{\max}(\sL_T^q)=\lambda_1(\sL_T^q)\geq \lambda_2(\sL_T^q) \geq 
\cdots \geq \lambda_n(\sL_T^q)=\lambda_{\min}(\sL_T^q)$. 
From Lemma \ref{lem:bapat-lal-pati}, it is easy to see  that $\lambda_{n-1}(\sL_T^q)>0$ for all $q \in \RR.$ 
Let the eigenvalues of $\ED_T$ be  
$\lambda_{\max}(\ED_T)=\lambda_1(\ED_T)\geq \lambda_2(\ED_T) \geq 
\cdots \geq \lambda_n(\ED_T)=\lambda_{\min}(\ED_T)$. 
Nagar and Sivasubramanian 
\cite[Remark 8 and Corollary 11]{mukesh-siva-immanantal_polynomial}    proved that 
both  $\sL_{T}^{q}|_{q=\alpha}$ and $\sL_{T}^{q}|_{q=-\alpha}$ have the same characteristic polynomial for  a tree $T$ 
and for all $\alpha\in \RR$. 
Thus, the multiset of eigenvalues of both the matrices  $\sL_{T}^{q}|_{q=\alpha}$ and $\sL_{T}^{q}|_{q=-\alpha}$ are equal.
This argument is used to prove the following lemma which will be used in the proof of Lemma \ref{lem:min_T1<2nd_min_T2}.

\begin{lemma}
	\label{lem:mult_small_ev}
	Let $T$ be a tree on $n$ vertices with $q$-Laplacian   $\sL_{T}^{q}$. 
	Then for all $q\in \RR\backslash\{0\}$, 
	the algebraic multiplicity of  $\lambda_{\min}(\sL_{T}^{q})$ as an eigenvalue of $\sL_{T}^{q}$ is 1.
\end{lemma}

\begin{proof}
	As for all $\alpha\in \RR$, the multiset of eigenvalues of both the matrices  $\sL_{T}^{q}|_{q=\alpha}$ and $\sL_{T}^{q}|_{q=-\alpha}$  are equal, it is sufficient to prove the corollary when $q\in \RR$ with $q>0$. 
	We first consider the case when $q\in \RR$ with $q\geq 1$, from Lemma \ref{lem:bapat-lal-pati}, 
	$\lambda_{\min}(\sL_{T}^{q})\leq 0$ and $\lambda_{n-1}(\sL_{T}^{q})>0$. 
	Thus, when $q\in \RR$ with $q\geq 1$, the algebraic multiplicity of  
	$\lambda_{\min}(\sL_{T}^{q})$ as an eigenvalue of $\sL_{T}^{q}$ is 1. 
	
	We next consider the case when $q\in \RR$ with $0<q<1$. 
	In this case each entry of the matrix  $\ED_T$ is   strictly positive. 
	Therefore by Perron's Theorem (see Horn and Johnson \cite[page 526 ]{horn-johnson-matrix-analysis}), 
	the algebraic multiplicity of $\lambda_{\max}(\ED_T)$ as an eigenvalue of $\ED_T$ is $1$ for all $q\in \RR$ with $0<q<1$.  
	In this case, from Lemma \ref{lem:bapat_ED_ev}, we get the following. 
	$$\lambda_{\max}(\ED_T)=\frac{1-q^2}{\lambda_{\min}(\sL_{T}^{q})} > 
	\frac{1-q^2}{\lambda_{n-1}(\sL_{T}^{q})}=\lambda_{2}(\ED_T).$$  
	Thus, $\lambda_{\min}(\sL_{T}^{q})< \lambda_{n-1}(\sL_{T}^{q})$ for all $q\in \RR$ with $0<q<1$ and the proof is complete.
\end{proof}

\begin{remark}
Lemma \ref{lem:mult_small_ev} generalizes the known theorem (see Godsil and Royle \cite{godsil-royle-book}) that 
$\lambda_2(L_G)>0$ for a connected graph $G$ as follows:  The algebraic multiplicity of $\lambda_{\min}(L_G)=0$ is $1$. 
Lemma \ref{lem:mult_small_ev} shows that for a tree $T$, the algebraic multiplicity of $\lambda_{\min}(\sL_T^q)$ for all $q\in \RR\backslash\{0\}$ is again $1$. 
It would be interesting to see if 
$\lambda_2(\sL_G^q)\neq \lambda_1(\sL_G^q)$ for all $q\in \RR\backslash\{0\}$ 
for all connected graphs $G$.  
\end{remark}

\section{General lemma}
\label{sec:gen_lemma}

We begin by recalling  
the following definition due to  Csikv{\'a}ri \cite{csikvari-poset2}. 
Let $Q_1=(V_1, E_1)$ and $Q_2=(V_2, E_2)$ be two trees on disjoint vertex sets. 
Let $v_1\in V_1$ and $v_2\in V_2$. 
Construct a new tree $T$ by moving all
the neighbours of $v_2$ to $v_1$  
and then deleting $v_2$.  
We thus treat $v_1$ and $v_2$ as a single vertex in $T$. 
The obtained tree $T$ is denoted by $Q_1^{v_1}:Q_2^{v_2}$ 
and has vertex set $V(Q_1^{v_1}:Q_2^{v_2})=V_1\cup V_2-\{v_2\}$ and 
edge set $E(Q_1^{v_1}:Q_2^{v_2})=E_1 \cup E_2$. 
We next illustrate this operation.

Let $T_1$ and $T_2$ be two trees with $n$ vertices 
such that $T_2$ covers $T_1$ in $\GTS_n$.  
Let $E(P_{k})$ be the edges on the path $P_{k}$ in $T_1$.  
Let $H_1$ and $H_2$ be the connected components of $T_1-E(P_{k})$  containing 
vertices $1$ and $k$  respectively. 
For the example of Figure \ref{fig:gts_example_lastchap},    
$H_1$ and $H_2$ are the subtrees of $T_1$ with vertex sets $X\cup\{1\}$ and $Y\cup\{k\}$ respectively. 
We also treat $H_2$ as a subtree of $T_2$ with vertex set $Y\cup\{1\}$. 
Therefore,  
$T_1=(H_1^1:P_k^1)^k:H_2^k$ and $T_2=(H_1^1:P_k^1)^1:H_2^1$. 
Thus, we obtain the following remark. We will use it in Section \ref{sec:F_T}. 
 \begin{remark}
 \label{rem:even_odd_no_vertices}
 Let $T_1$ and $T_2$ be two trees on $n$ vertices such that  $T_2$ covers $T_1$ in $\GTS_n$. 
 Then,  $|P_k|+|H_1|+|H_2|=n+2$. 
 When $n$ is an even number then either all three subtrees $P_k$, $H_1$ and 
 $H_2$ have an even number of vertices or exactly one subtree has an even number of vertices. 
 Similarly, when $n$ is an odd number then either all three subtrees have an odd number of vertices 
 or exactly one subtree has an odd number of vertices.
 \end{remark}

Let $T$ be a tree on $n$ vertices with $q$-Laplacian $\sL_{T}^{q}$.  
We consider  
the characteristic polynomial 
$f^{\sL_{T}^{q}}(q,x) = \det (xI -\sL_{T}^{q})$.  
This is a bivariate polynomial. 
For a fixed vertex $v\in T$, 
let  $\sL_{T}^{q}|v$ be  the submatrix 
obtained by deleting $v$-th row and $v$-th column from $\sL_{T}^{q}$.
 Let $f^{\sL_{T}^q|v}(q,x)=\det(xI-\sL_{T}^{q}|v)$.

The following lemma is called the general lemma which  was proved 
by  Csikv{\'a}ri \cite{csikvari-poset2} for  graph polynomials in one variable $x$. 
We will apply it to the  
characteristic polynomial of $\sL_{T}^{q}$ which is a bivariate polynomial.  
Since the proof is identical to that of 
Csikv{\'a}ri \cite[Theorem 5.1]{csikvari-poset2}, 
we omit it and merely state the result. 
\begin{lemma}
\label{lem:gen_lemma}
Let $Q_1$ and $Q_2$ be two trees. 
For any two fixed vertices $v_1\in Q_1$ 
and $v_2\in Q_2$, let $T=Q_1^{v_1}:Q_2^{v_2}$. Suppose 
\begin{eqnarray}
f^{\sL_{T}^q}(q,x) 
& = & y_1f^{\sL_{Q_1}^{q}}(q,x)f^{\sL_{Q_2}^{q}}(q,x)
+ y_2f^{\sL_{Q_1}^{q}}(q,x)f^{\sL_{Q_2}^{q}|v_2}(q,x) 
+ y_2f^{\sL_{Q_1}^{q}|v_1}(q,x)f^{\sL_{Q_2}^{q}}(q,x) 
\nonumber \\ & & 
+ y_3f^{\sL_{Q_1}^{q}|v_1}(q,x) f^{\sL_{Q_2}^{q}|v_2}(q,x), 
\label{eq:gen_lem_condition}
\end{eqnarray} 
where $y_1$, $y_2$, $y_3$ are bivariate rational functions of $q$ and 
$x$. 
Let $T_1$ and $T_2$ be two trees with $n$ vertices such that 
$T_2$ covers  $T_1$ in $\GTS_n$ and 
$y_2f^{\sL_{K_2}^{q}}(q,x)+y_3f^{\sL_{K_2}^{q}|1}(q,x)\neq 0$, 
where $K_2$ is the complete graph on $2$ vertices. Then, 
\begin{eqnarray}
f^{\sL_{T_1}^{q}}(q,x)-f^{\sL_{T_2}^{q}}(q,x) & = &
y_4\left[y_2f^{\sL_{P_k}^{q}}(q,x)+y_3f^{\sL_{P_k}^{q}|1}(q,x)\right]
\left[y_2f^{\sL_{H_1}^{q}}(q,x)+y_3f^{\sL_{H_1}^{q}|1}(q,x)\right]
\nonumber \\
& &  \times  \left[y_2f^{\sL_{H_2}^{q}}(q,x)+y_3f^{\sL_{H_2}^{q}|1}(q,x)\right],
\nonumber  \\
\mbox{ where } & & y_4  =  \frac{ \left[f^{\sL_{P_3}^{q}|1}(q,x)-f^{\sL_{P_3}^{q}|2}(q,x)\right]}
{\left[y_2f^{\sL_{K_2}^{q}}(q,x)+y_3f^{\sL_{K_2}^{q}|1}(q,x)\right]^2}.
\label{eq:y_4}
\end{eqnarray}
\end{lemma}

Let $ Q_1$ and $ Q_2$ be two trees and let $T=Q_1^{v_1}:Q_2^{v_2}$. 
Let $[m]=\{1,2,\ldots,m\}$ be the vertex set of $T$ and let $v_1=r=v_2$, where 
 $\{1,2,\ldots,r \}$ and  
$\{r,r+1,\ldots,m \}$ be the vertex sets of $Q_1$ and $Q_2$ respectively.
Let $\sL_{Q_1}^{q}$, $\sL_{Q_2}^{q}$ 
and $\sL_{T}^q$ be the  $q$-Laplacians of $Q_1$, $Q_2$ 
and $T$ respectively. 
We extend the notion of  $\sL_{T}^{q}|v$ to an arbitrary subset of $[m]$. 
Let $S\subseteq [m]$ and let $S'=[m]-S$. 
Let  $\sL_{T}^{q}|S$ be the submatrix of $\sL_{T}^{q}$ induced on the  rows and 
columns with indices in  the set $S'$.
We need the following lemma to obtain the   
rational functions $y_1$, $y_2$, $y_3$ and $y_4$.  

\begin{lemma}
\label{lem:condition_gen_lemma}
Let $\sL_{Q_1}^{q}$, $\sL_{Q_2}^{q}$, $\sL_{T}^q$, $\sL_{Q_1}^{q}|v_1$ 
and $\sL_{Q_2}^{q}|v_2$ be the matrices as defined in the above paragraph. 
Let $f^{\sL_{Q_1}^{q}}(q,x)$, $f^{\sL_{Q_2}^{q}}(q,x)$,  
$f^{\sL_{T}^{q}}(q,x)$, $f^{\sL_{Q_1}^{q}|v_1}(q,x)$ 
and $f^{\sL_{Q_2}^{q}|v_2}(q,x)$ denote the characteristic polynomials of 
$\sL_{Q_1}^{q}$, $\sL_{Q_2}^{q}$,  $\sL_{T}^q$, $\sL_{Q_1}^{q}|v_1$ and $\sL_{Q_2}^{q}|v_2$ respectively. 
Then,  these polynomials satisfy the condition given in  \eqref{eq:gen_lem_condition}. 
\end{lemma}
 
\begin{proof}   
Let $l_{r,r}^{}$ denote the 
$(r,r)$-th entry in the bivariate polynomial matrix $xI-\sL_{T}^{q}$. 
Therefore, 
$l_{r,r}^{}=x-(1+q^2(d_{v_1}+d_{v_2}-1))$, 
where  $d_{v_1}$ and $d_{v_2}$ are the degrees of the vertices $v_1$ 
and $v_2$ in $Q_1$ and $Q_2$ respectively. 
Let  $\mu=(\mu_i)_{1\leq i \leq r-1}$ and 
$\nu=(\nu_j)_{r+1\leq j \leq m}$ be two column  vectors such that $\mu_i=q$ if $i$ is a 
neighbour of $r$ in $Q_1$ and $\mu_i=0$ otherwise. Similarly, $\nu_j=q$ if $j$ is a 
neighbour of $r$ in $Q_2$ and $\nu_j=0$ otherwise.
Let $R=\{1,2,\ldots,r\}$, $S=\{1,2,\ldots,r-1\}$, $R'=[m]-R$ and $S'=[m]-S$. 
Therefore, 
\begin{equation}
\label{eqn:matrix_xI-sL_T}
xI-\sL_{T}^{q}
=
\left[
\begin{array}{c c c}
N & \mu & 0  \\
\mu^t  & l_{r,r}^{} & \nu^t \\
0 & \nu & M 
\end{array}
\right],
\mbox{ where } M=(xI-\sL_{T}^{q}|R) \mbox{ and } N=(xI-\sL_{T}^{q}|S').
\end{equation}  

Clearly  $\det(M)=f^{\sL_{Q_2}^{q}|v_2}(q,x)$ and $\det(N)=f^{\sL_{Q_1}^{q}|v_1}(q,x)$. 
Let $ M'  =(xI-\sL_{T}^{q}|R') $ and $N'=(xI-\sL_{T}^{q}|S)$.  Then, it is simple to see the following. 
\begin{eqnarray}
\det (M') & = & 
\det\left[ 
 \begin{array}{ c c}
 N & \mu \\
 \mu^t & x-(1+q^2(d_{v_1}+d_{v_2}-1))
 \end{array}
 \right]
 =  f^{\sL_{Q_1}^{q}}(q,x)-q^2d_{v_2}f^{\sL_{Q_1}^{q}|v_1}(q,x). \label{eq:det_M_i} \\  \nonumber\\
\det(N') & = & 
\det\left[
\begin{array}{ c c}
x-(1+q^2(d_{v_1}+d_{v_2}-1)) & \nu^t \\
\nu & M
\end{array}
\right]
=f^{\sL_{Q_2}^{q}}(q,x)-q^2d_{v_1}f^{\sL_{Q_2}^{q}|v_2}(q,x).
 \label{eq:det_N'_i}
 \end{eqnarray}    
 
From \eqref{eqn:matrix_xI-sL_T}, when we  expand  $\det(xI-\sL_{T}^{q})$ with respect to the $r$-th row, 
we get  
\begin{eqnarray} 
f^{\sL_{T}^{q}}(q,x) & = &  \det (xI-\sL_{T}^{q})  
= \det (M') \det (M) +\det (N) \det (N') 
 -  l_{r,r}^{}\det (N) \det (M)  \nonumber \\
& = & \left[f^{\sL_{Q_1}^{q}}(q,x)-q^2d_{v_2}f^{\sL_{Q_1}^{q}|v_1}(q,x)\right]
f^{\sL_{Q_2}^{q}|v_2}(q,x)\nonumber \\ 
& & +f^{\sL_{Q_1}^{q}|v_1}(q,x)
\left[f^{\sL_{Q_2}^{q}}(q,x)-q^2d_{v_1}f^{\sL_{Q_2}^{q}|v_2}(q,x)\right] \nonumber \\ 
& & - \left(x-(1+q^2(d_{v_1}+d_{v_2}-1))\right)f^{\sL_{Q_1}^{q}|v_1}(q,x)f^{\sL_{Q_2}^{q}|v_2}(q,x)  \nonumber \\
& = & f^{\sL_{Q_1}^{q}}(q,x)f^{\sL_{Q_2}^{q}|v_2}(q,x)
+f^{\sL_{Q_2}^{q}}(q,x)f^{\sL_{Q_1}^{q}|v_1}(q,x) \nonumber \\
& & - (x-1+q^2)f^{\sL_{Q_1}^{q}|v_1}(q,x)f^{\sL_{Q_2}^{q}|v_2}(q,x), 
\label{eq:proved_condition} 
\end{eqnarray} 
where the third equality follows from \eqref{eq:det_M_i} 
and \eqref{eq:det_N'_i}. The last equality follows 
by doing simple manipulations completing the proof.
\end{proof}


Let $T$ be a tree with $q$-Laplacian $\sL_T^q$. For a fixed $q\in \RR\backslash \{0\}$ and 
for a fixed vertex $v\in T$, define the auxiliary polynomial  
\begin{equation}
\label{eq:def_F_T^v}
\sF_{T}^{v}(q,x)=f^{\sL_{T}^{q}}(q,x) -(x+q^2 -1)f^{\sL_{T}^{q}|v}(q,x).
\end{equation}

We recall that  $P_k$, $H_1$ and $H_2$ 
are the subtrees of $T_1$ and $T_2$, where $T_2$ covers $T_1$ in $\GTS_n$. 
From  \eqref{eq:gen_lem_condition} and 
\eqref{eq:proved_condition}, 
we get the rational functions $y_1=0$, $y_2=1$, $y_3=-(x-1+q^2)$. 

\begin{theorem}
\label{thm:result_by_gen_lemma}
Let $T_1$ and $T_2$ be two trees on $n$ vertices with $q$-Laplacians 
$\sL_{T_1}^{q}$ and $\sL_{T_2}^{q}$   respectively.  
Let $T_2$ cover $T_1$ in $\GTS_n$. Then, for all $q\in \RR\backslash \{0\}$ 
$$f^{\sL_{T_1}^{q}}(q,x)- f^{\sL_{T_2}^{q}}(q,x)
= -\frac{1}{q^2x} \sF_{P_k}^{1}(q,x)\sF_{H_1}^{1}(q,x)\sF_{H_2}^{1}(q,x),$$
where $\sF_{P_k}^{1}(q,x)$, $\sF_{H_1}^{1}(q,x)$ and 
$\sF_{H_2}^{1}(q,x)$ are the polynomials  defined in \eqref{eq:def_F_T^v}.
\end{theorem}
\begin{proof}
From Lemma \ref{lem:condition_gen_lemma}, 
we get $y_1=0$, $y_2=1$, $y_3=-(x-1+q^2)$. Therefore when $q\in \RR\backslash \{0\}$ 
$$y_2f^{\sL_{K_2}^{q}}(q,x)+y_3f^{\sL_{K_2}^{q}|1}(q,x) 
=  (x-1)^2-q^2-(x-1+q^2)(x-1)=-q^2x \neq 0.  $$
 
Note that $ f^{\sL_{P_3}^{q}|1}(q,x)-f^{\sL_{P_3}^{q}|2}(q,x) =  -q^2x.$
Thus, by \eqref{eq:y_4}, we have $\displaystyle y_4=-1/q^2x$. 
Using Lemma \ref{lem:gen_lemma}, the proof is complete.
\end{proof}

\section{Interlacing of eigenvalues of $\sL_{T}^{q}$}
\label{sec:interlacing}

Let $A$ be an $n \times n$ real symmetric matrix. 
We order  the eigenvalues of $A$  as   
$\lambda_1(A)\geq \lambda_2(A)\geq \cdots \geq \lambda_n(A).$
We need the following lemma from  Molitierno \cite[Theorem 1.2.8 and Corollary 1.2.11]{molitierno-book}.

\begin{lemma}
\label{rem:ev_A_B}
Let $A, B$ be two $n\times n$ real symmetric matrices 
with $B=zz^t$ for some column vector $z\in \RR^n$. Then,  
$$\lambda_{1}(A+B)\geq \lambda_{1}(A) 
\geq \cdots \geq \lambda_{n}(A+B)\geq \lambda_{n}(A). $$
\end{lemma}

We have two interlacing lemmas about the eigenvalues of $\sL_T^q$ when $q\in \RR$ with either $|q|\leq 1$  or $|q|>1$. 
The following result is an interlacing lemma about the eigenvalues of $\sL_T^q$  when $|q|\leq 1$.
 
\begin{lemma}
\label{lem:q_interlacing}
Let  $T$ be a tree on $n$ vertices with a leaf vertex $l$. Let $T'=T-\{l\}$.  
Let $\sL_{T}^{q}$ and $\sL_{T'}^{q}$ be the $q$-Laplacians of 
$T$ and $T'$ respectively.
Then, for $q\in \RR$ with $|q|\leq 1$,  
$$\lambda_1(\sL_{T}^{q})\geq \lambda_1(\sL_{T'}^q) \geq  
\cdots \geq \lambda_{n-1}(\sL_{T'}^{q}) \geq \lambda_n(\sL_{T}^{q})\geq 0.$$
\end{lemma}
\begin{proof} 
We can without loss of generality assume in $T$ that $l=n$ is the deleted leaf vertex with neighbour $n-1$. 
Let $e_{n-1}=[0,0,\hdots,0,1]^t\in \RR^{n-1}$. 
Then, 
\begin{eqnarray*}
\sL_{T}^{q} & = & \left[
\begin{array}{c c}
q^{2}e_{n-1}e^{t}_{n-1}+\sL_{T'}^{q} & -qe_{n-1}  \\
-qe^{t}_{n-1} & 1  
\end{array}
\right] = 
\left[
\begin{array}{c c}
\sL_{T'}^{q} & 0  \\
0 & 0  
\end{array}
\right] +
\left[
\begin{array}{c c}
q^{2}e_{n-1}e^{t}_{n-1} & -qe_{n-1}  \\
-qe^{t}_{n-1} & 1  
\end{array}
\right] \\
& = &
\left[
\begin{array}{c c}
\sL_{T'}^{q} & 0  \\
0 & 0  
\end{array}
\right] + zz^t, \mbox{ where }  z=[0,0,\ldots,0,-q,1]^t \in \RR^n.
\end{eqnarray*}

By Lemma \ref{lem:bapat-lal-pati},  
both  $\sL_{T}^{q}$ and 
$\sL_{T'}^{q}$  are positive semidefinite when $|q|\leq 1$.  
Therefore, all the eigenvalues of $\sL_{T}^{q}$ and $\sL_{T'}^{q}$ are non-negative.
Thus, by using Lemma \ref{rem:ev_A_B}, we have 
$\lambda_1(\sL_{T}^{q})\geq \lambda_1(\sL_{T'}^q)\geq 
\cdots \geq\lambda_{n-1}(\sL_{T'}^{q})\geq 
 \lambda_n(\sL_{T}^{q}) \geq  0$ 
 completing the proof.
\end{proof}


From Lemma \ref{lem:bapat-lal-pati}, 
$\sL_T^q$ is not  positive semidefinite  when $q\in \RR$ with $|q|>1$.
We prove the following partial interlacing lemma about the eigenvalues of  $\sL_T^q$ when $|q|>1$. 

\begin{lemma}
\label{lem:q_partial_interlacing}
Let  $T$ be a tree on $n$ vertices with a leaf vertex $l$. Let $T'=T-\{l\}$.  
Let $\sL_{T}^{q}$ and $\sL_{T'}^{q}$ be the $q$-Laplacians of 
$T$ and $T'$ respectively.
Then, for $q\in \RR$ with $|q|> 1$,   
$$\lambda_1(\sL_{T}^{q})\geq \lambda_1(\sL_{T'}^q) \geq  
\cdots \geq \lambda_{n-2}(\sL_{T'}^{q}) \geq \lambda_{n-1}(\sL_{T}^{q}) > 0 > 
\lambda_{n}(\sL_{T}^{q}) \geq \lambda_{n-1}(\sL_{T'}^{q}).$$
\end{lemma}
\begin{proof}
As done in Lemma \ref{lem:q_interlacing}, we assume the vertex $l=n$ and 
that its neighbour is vertex $n-1$. Thus,  we obtain  $\sL_{T}^{q}=  A+B$, where 
\begin{equation*}
A= 
\left[
\begin{array}{c c}
\sL_{T'}^{q} & 0  \\
0 & 0  
\end{array}
\right] \mbox{ and }   
B=zz^t \mbox{ with }  z=[0,0,\ldots,0,-q,1]^t \in \RR^n.
\end{equation*}

By Lemma \ref{lem:bapat-lal-pati},   both $\sL_{T}^q$ and 
$\sL_{T'}^q$ have exactly one negative eigenvalue and 
both $\sL_T^q$ and $\sL_{T'}^q$ are invertible when $|q|>1$. 
Therefore, the eigenvalues of $A$  are the following:  
$\lambda_1(\sL_{T'}^q) \geq  \cdots \geq
\lambda_{n-2}(\sL_{T'}^q) >  0 > \lambda_{n-1}(\sL_{T'}^q).$
Thus, by Lemma \ref{rem:ev_A_B}, $\lambda_1(\sL_{T}^{q})\geq \lambda_1(\sL_{T'}^q) \geq  
\cdots \geq \lambda_{n-1}(\sL_{T}^{q}) > 0 >  \lambda_{n}(\sL_{T}^{q}) 
\geq \lambda_{n-1}(\sL_{T'}^{q})$ completing the proof.
\end{proof}


Let $T$ be a tree on $n$ vertices with $q$-Laplacian $\sL_T^q$. 
Let  $\lambda_{\max}(\sL_T^q)$,  $\lambda_{\min}(\sL_T^q)$ and $\lambda_{a}(\sL_T^q)=\lambda_{n-1}(\sL_T^q)$
denote the largest, the smallest and the second smallest eigenvalues of $\sL_T^q$ respectively.
We need the following corollaries of Lemmas \ref{lem:q_interlacing} and \ref{lem:q_partial_interlacing}.
\begin{corollary}
\label{cor:subtree_ev}
Let  $T$ be a tree with $n$ vertices and let $T'$ be a subtree of $T$. 
Let $\sL_{T}^{q}$ and $\sL_{T'}^{q}$ be the $q$-Laplacians of $T$ and $T'$ respectively. 
Then, for all $q\in \RR$, $\lambda_{\max}(\sL_{T'}^q)\leq \lambda_{\max}(\sL_{T}^{q}).$
\end{corollary}
\begin{proof}
Let $T'=T_0$ be the given subtree of $T$ and let $m$ be the number of vertices in $T_0$. 
Construct a new tree $T_1$ by adding a leaf vertex in $T_0$ 
such that $T_1$ is again a subtree of $T$.  
Thus, from Lemmas \ref{lem:q_interlacing} and  \ref{lem:q_partial_interlacing}, 
$\lambda_{\max}(\sL_{T_0}^q) \leq \lambda_{\max}(\sL_{T_1}^{q})$ for all $q\in \RR$.  
By repeating this process we get a sequence of subtrees $T'=T_0$, $T_1$,  $\ldots$, $T_{n-m}=T$ of $T$ with   
$\lambda_{\max}(\sL_{T'}^q)=\lambda_{\max}(\sL_{T_0}^q)  \leq \lambda_{\max}(\sL_{T_1}^{q}) 
\leq \cdots \leq \lambda_{\max}(\sL_{T_{n-m}}^q)=\lambda_{\max}(\sL_{T}^{q})$ 
 and hence, completing the  proof.
\end{proof}


By using similar argument as done in the proof of Corollary \ref{cor:subtree_ev}, 
we see that     $\lambda_{a}(\sL_{T'}^q)\geq \lambda_{a}(\sL_{T}^{q})$, 
where $T'$ is a subtree of $T$.
Thus, we obtain the following corollary of Lemmas \ref{lem:q_interlacing} and  
\ref{lem:q_partial_interlacing}. 

\begin{corollary}
\label{cor:ev_n-1_subtrees}
Using the notations of Theorem \ref{thm:result_by_gen_lemma}, 
let $P_k$, $H_1$ and $H_2$ be the subtrees of both $T_1$ and $T_2$. 
Then, for all $q\in \RR$, we have   $\max\left(\lambda_{a}(\sL_{T_1}^{q}),
\lambda_{a}(\sL_{T_2}^{q})\right) \leq \min \left(\lambda_{a}(\sL_{P_k}^{q}),
\lambda_{a}(\sL_{H_1}^{q}),\lambda_{a}(\sL_{H_2}^{q}) \right).$
\end{corollary}

\section{Auxiliary polynomial $\sF_{T}^{v}(q,x)$}
\label{sec:F_T}

Let $T$ be a tree on $n$ vertices with  $q$-Laplacian $\sL_{T}^{q}$. 
We recall the definition of the polynomial $\sF_{T}^{v}(q,x)$ defined in \eqref{eq:def_F_T^v} as follows: 
$$\sF_{T}^{v}(q,x)=f^{\sL_{T}^{q}}(q,x)-(x-1+q^2)f^{\sL_{T}^{q}|v}(q,x)
=f^{\sL_{T}^{q}}(q,x)-(x-\det(\sL_T^q))f^{\sL_{T}^{q}|v}(q,x).$$

We think of $\sF_{T}^{v}(q,x)$ as a univariate polynomial once a real value for $q$ is assigned. 
From Theorem \ref{thm:result_by_gen_lemma}, 
we recall that when a tree $T_2$ covers an another tree $T_1$ in $\GTS_n$, 
$f^{\sL_{T_1}^{q}}(q,x)- f^{\sL_{T_2}^{q}}(q,x)$ is a product of three auxiliary 
polynomials of subtrees $P_k$, $H_1$, and $H_2$ of $T_1$ and $T_2$. 
Therefore the roots of this difference polynomial is the multiset  
union of the roots of auxiliary polynomials of these subtrees $P_k$, $H_1$, and $H_2$. 
Thus, to prove Theorem \ref{thm:main_result}, we need to determine 
the location of all these roots and decide the sign of $f^{\sL_{T_1}^{q}}(q,x)- f^{\sL_{T_2}^{q}}(q,x)$ 
for a fixed $q\in \RR\backslash\{0\}$ and 
when $x\in [(-\infty,t_1)\cup (t_2,\infty)]-\{0\}$, 
where $t_1=\max(\lambda_{a}(\sL_{T_1}^q),\lambda_{a}(\sL_{T_2}^q))$ and 
$t_2=\min(\lambda_{\max}(\sL_{T_1}^q),\lambda_{\max}(\sL_{T_2}^q))$.  
\begin{lemma}
	\label{lem:deg_aux_poly}
	For a tree $T$ on $n>1$ vertices with $q$-Laplacian $\sL_T^q$, 
	the degree of $\sF_{T}^{v}(q,x)$  is $n-1$ and the coefficient of $x^{n-1}$ in $\sF_{T}^{v}(q,x)$ is $-q^2d_v$. 
	Further, for all $q\in \RR\backslash \{0\}$, zero is a root of $\sF_{T}^{v}(q,x)$.
\end{lemma}
\begin{proof} 
	Without loss of generality we assume that the first row of $\sL_{T}^{q}$ is indexed by the vertex $v$. 
	Let $\alpha=(\alpha_i)_{2\leq i \leq n}$ be a column vector such that $\alpha_i=q$ if $v$ is adjacent to the vertex $i$ and $\alpha_i=0$ otherwise. 
	Therefore 
	\begin{eqnarray}
	f^{\sL_{T}^{q}}(q,x) & = &  
	\det \left[
	\begin{array}{c c}
	x-1-q^2(d_v-1) & \alpha^t \\
	\alpha & xI-\sL_{T}^{q}|v 
	\end{array}
	\right]   \nonumber \\  
	& = & 
	\det\left[
	\begin{array}{c c}
	x-1+q^2 & \alpha^t \\
	0 & xI-\sL_{T}^{q}|v 
	\end{array}
	\right]
	+
	\det\left[
	\begin{array}{c c}
	-q^2d_v & \alpha^t \\
	\alpha & xI-\sL_{T}^{q}|v 
	\end{array}
	\right]   \nonumber \\
	& = & 
	(x-1+q^2)f^{\sL_{T}^{q}|v}(q,x) + 
	\det\left[
	\begin{array}{c c}
	-q^2d_v & \alpha^t \\
	\alpha & xI-\sL_{T}^{q}|v 
	\end{array}
	\right]
	\label{eqn:q-Lapl_poly}
	\end{eqnarray}
	
	From  the definition of $\sF_{T}^{v}(q,x)$ and by using  \eqref{eqn:q-Lapl_poly}, 
	it is easy to see that the degree of the polynomial $\sF_{T}^{v}(q,x)$ in the variable $x$ is $n-1$ 
	and the coefficient of $x^{n-1}$ in $\sF_{T}^{v}(q,x)$ is $-q^2d_v$. 
	
	Nagar and Sivasubramanian \cite[Remark 13]{mukesh-siva-immanantal_polynomial}  proved that  
	$f^{\sL_{T}^{q}|v}(q,0)=(-1)^{n-1}$.  Thus, when $x=0$ from  Lemma \ref{lem:bapat-lal-pati},  
	we get the following. 
	$$\sF_{T}^{v}(q,0)=f^{\sL_{T}^{q}}(q,0) -(q^2 -1)f^{\sL_{T}^{q}|v}(q,0)
	=(-1)^n\det(\sL_T^q)-(q^2-1)(-1)^{n-1}=0.$$ 
	Thus, zero is a root of $\sF_{T}^{v}(q,x)$ and hence the proof is complete.
\end{proof} 

\vspace{2mm} 

For a fixed $q\in \RR\backslash \{0\} $, from Lemma \ref{lem:deg_aux_poly}, 
it is easy to see that for large positive $x$ the function $\sF_{T}^{v}(q,x)$ is negative. 
When $|q|\leq 1$, we next prove that the multiplicity of zero as a root of 
$\sF_{T}^{v}(q,x)$ is one and determine the location of its $n-2$ non-zero roots. 
Let the eigenvalues of $\sL_{T}^{q}|v$ be  
$\lambda_1(\sL_{T}^q|v)\geq \lambda_2(\sL_{T}^q|v)\geq  \cdots \geq \lambda_{n-1}(\sL_{T}^q|v)$. 
Let $\sigma(\sL_{T}^q|v)=\{\lambda_1(\sL_{T}^q|v), \ldots, \lambda_{n-1}(\sL_{T}^q|v) \}$ 
be the multiset of eigenvalues of  $\sL_T^q|v$.
From \cite[Corollary 12]{mukesh-siva-immanantal_polynomial}, 
it is simple to see that all the eigenvalues of $\sL_{T}^{q}|v$ are non-negative. 
Let $\sigma(\sL_{T}^q)=\{\lambda_1(\sL_{T}^q), \ldots, \lambda_{n-1}(\sL_{T}^q) \}$ 
be the multisets of the $n-1$ largest eigenvalues of $\sL_T^q$.  
Motivated  by  Csikv{\'a}ri \cite[ the third part of Theorem 7.3]{csikvari-poset2}  
when $q\in \RR$ with $|q|\leq 1$, we obtain the following result. 
\begin{lemma}
	\label{lem1:F_T_q|<1}
	Let $T$ be a tree on $n>1$ vertices and 
	let $v\in T$.  
	Then,  for a fixed $q\in \RR\backslash \{0\} $ with $|q|\leq 1$, 
	each  non-zero root of $\sF_{T}^{v}(q,x)$ lies  in the interval   
	$[\lambda_{a}(\sL_{T}^{q}),\lambda_{\max}(\sL_{T}^{q})]$. 
\end{lemma}

\begin{proof}     
	Using the Interlacing Theorem for eigenvalues of symmetric matrices 
	(see Godsil and Royle \cite[page 193]{godsil-royle-book}),  we get the following. 
	\begin{equation}
	\label{eqn:interlac_symm}
	\lambda_1(\sL_{T}^q)\geq \lambda_1(\sL_{T}^q|v) \geq  \cdots \geq \lambda_{i}(\sL_{T}^q) \geq 
	\lambda_i(\sL_{T}^q|v)  \geq \cdots \geq \lambda_{n-1}(\sL_{T}^q|v) \geq \lambda_n(\sL_{T}^q).
	\end{equation}
	
	We break the proof into two cases when $1-q^2 \geq \lambda_{n-1}(\sL_{T}^q)$ and 
	when $1-q^2 < \lambda_{n-1}(\sL_{T}^q).$ 
	For both the cases, the proof is identical. 
	Thus, we only consider the case when  $1-q^2 \geq \lambda_{n-1}(\sL_{T}^q)$.  

	 Firstly we assume that 	 $\sigma(\sL_{T}^q) \Cap \left[\sigma(\sL_{T}^q|v)\Cup\{1-q^2\}\right]=\emptyset$, 
	 where $\Cap$ and $\Cup$ denote the multiset intersection and multiset union respectively.  
	Therefore from \eqref{eqn:interlac_symm}, 
	the multiplicity of each eigenvalue of $\sL_{T}^q$ and $\sL_{T}^q|v$ is one and we get 
	\begin{equation}
	\label{eqn:interlac_ev}
	\lambda_1(\sL_{T}^q)> \lambda_1(\sL_{T}^q|v)  > \cdots 
	>\lambda_{n-1}(\sL_{T}^q) > \lambda_{n-1}(\sL_{T}^q|v) \mbox{ with } 1-q^2 > \lambda_{n-1}(\sL_{T}^q).
	\end{equation}
	
	It is easy to see that 
	\begin{eqnarray}
	\lambda_{\max}(\sL_{T}^q) & \geq  & \frac{\sum_{i=0}^{n}\lambda_i(\sL_{T}^q)}{n}  
	= \frac{\Tr(\sL_{T}^q)}{n}  =
	\frac{\sum_{v\in T} 1+q^2(d_v-1)}{n}=\frac{n+q^2(n-2)}{n} 
	\nonumber \\
	& \geq & 1-q^2. 
	\label{eqn:max_ev_det}
	\end{eqnarray}
	Thus, from \eqref{eqn:interlac_ev} and \eqref{eqn:max_ev_det},  
	$1-q^2$
	is sandwiched between two eigenvalues of $\sL_T^q$. 
	Thus,  for some $i$ with $1 \leq i \leq n-2$, we get 
	 $$\lambda_{i}(\sL_{T}^q) > \lambda_{i}(\sL_{T}^q|v) > \lambda_{i+1}(\sL_{T}^q) \mbox{ with }   
	\lambda_{i}(\sL_{T}^q) > m_1 \geq m_2 > \lambda_{i+1}(\sL_{T}^q),$$ 
	where $m_1=\max \left(1-q^2, \lambda_{i}(\sL_{T}^q|v) \right)$ and 
	$m_2=\min \left(1-q^2, \lambda_{i}(\sL_{T}^q|v) \right)$. 
	Recall for a fixed $q\in \RR\backslash \{0\}$,  
	the polynomial $\sF_T^v(q,x)$ is univariate in the variable $x$. 
	Thus, by the intermediate value theorem (IVT henceforth), 
	as both the quantities $\sF_T^v(q,m_1)=f^{\sL_T^q}(q,m_1)$ and 
	$\sF_T^v(q,m_2)=f^{\sL_T^q}(q,m_2)$ have same sign,  either both are  positive or both are negative.
	
	From Lemma \ref{lem:deg_aux_poly}, we recall that for large positive 
	$x \in \RR$, $\sF_{T}^{v}(q,x)$ is negative. 
	 When $1<i$,    $\lambda_1(\sL_{T}^q)>\lambda_1(\sL_{T}^q|v)>m_1$ and 
	 the coefficient of $x^{n-1}$ in $f^{\sL_T^q|v}(q,x)$ is positive.  
	 Therefore  $\sF_{T}^{v}(q,\lambda_1(\sL_{T}^q)) 
	 = -(\lambda_1(\sL_{T}^q)-1+q^2)f^{\sL_{T}^{q}|v}(q,\lambda_1(\sL_{T}^q))$ is negative. 
	Similarly  when $2<i$, the coefficient of $x^{n}$ in $f^{\sL_T^q}(q,x)$ is positive and  $\lambda_1(\sL_{T}^q)>\lambda_1(\sL_{T}^q|v)>\lambda_2(\sL_{T}^q)>m_1$. Thus, by the IVT,  $\sF_{T}^{v}(q,\lambda_1(\sL_{T}^q|v))=f^{\sL_T^q}(q,\lambda_1(\sL_{T}^q|v))$ is again negative.   
	By using identical arguments when $x \in \{
	 \lambda_2(\sL_{T}^q),\lambda_2(\sL_{T}^q|v)\}$ with $\lambda_2(\sL_{T}^q|v)>m_1$, 
	 $\sF_{T}^{v}(q,x)$ is positive. 
	 $\sF_{T}^{v}(q,x)$ is negative if $j$ is odd and positive if $j$ is even. 
	 But when $j=i$, both the quantities $\sF_T^v(q,m_1)$ and $\sF_T^v(q,m_2)$ have the same sign. 
	Therefore, it is easy to extend this for all $j$ with $1\leq j\leq n-1$. 
	Thus, when $x \in  \left\{\lambda_j(\sL_{T}^q), \lambda_{j}(\sL_{T}^q|v)\right\}$ $\sF_{T}^{v}(q,x)$ 
	is negative if $j$ is odd and positive if $j$ is even. 
	Hence, there must be a root of $\sF_{T}^{v}(q,x)$ in the interval  
	$[\lambda_{j+1}(\sL_{T}^q),\lambda_j(\sL_{T}^q|v)]$ for  $j=1,2,\ldots,n-2$. 
	See Example \ref{eg:eg_graphs} for better interpretation.

	Secondly we assume that 
	$\sigma(\sL_{T}^q) \Cap \left[\sigma(\sL_{T}^q|v)\Cup \{1-q^2\}\right]
	=\{ \lambda_1,\lambda_2, \ldots,\lambda_c\}$. 
	Clearly  $\lambda_1$, $\lambda_2$, $\ldots$, $\lambda_c$ 
	are the roots of $\sF_{T}^{v}(q,x)$ and 
	these lie in the interval $[\lambda_{a}(\sL_{T}^q),\lambda_{\max}(\sL_{T}^q)]$. 
	From \eqref{eqn:interlac_symm}, it is clear that the multiplicity of 
	each eigenvalue not containing in  $\{ \lambda_1, \ldots,\lambda_c\}$ of $\sL_{T}^q$ 
	and $\sL_{T}^q|v$ is one. 
	Therefore the remaining eigenvalues of $\sL_{T}^q$ and 
	$\sL_{T}^q|v$ satisfy an identical relation as given in \eqref{eqn:interlac_ev}. 
	Thus, by using similar arguments as done in the above paragraph, 
	it is easy to determine the location of the remaining $n-c-2$ roots of 
	$\sF_{T}^{v}(q,x)$ by locating the roots of the following polynomial. 
	$$\frac{\sF_{T}^{v}(q,x)}{\prod_{i=1}^{c}(x-\lambda_i)}
	=\frac{f^{\sL_{T}^{q}}(q,x)}{\prod_{i=1}^{c}(x-\lambda_i)}
	-\frac{(x-1+q^2)f^{\sL_{T}^{q}|v}(q,x)}{\prod_{i=1}^{c}(x-\lambda_i)}.$$ 
		
Thus, for $q\in \RR\backslash\{0\}$ with $ |q|\leq 1$, 
all the non-zero roots of $\sF_{T}^{v}(q,x)$ lie in the interval 
$[\lambda_{a}(\sL_{T}^q),\lambda_{\max}(\sL_{T}^q)]$. 
The proof is complete. 
\end{proof}

\vspace{2mm}
When $q\in \RR$ with $|q|>1$, the proof of the following lemma is identical to the proof of Lemma \ref{lem1:F_T_q|<1}.
\begin{lemma}
	\label{lem2:F_T_q|>1}
	Let $T$ be a tree on $n>1$ vertices and 
	let $v\in T$.  
	Then,  for all $ q\in \RR$ with $|q|>1$, each non-zero roots of $\sF_{T}^{v}(q,x)$  lies in the interval   $[\lambda_{a}(\sL_{T}^{q}),\lambda_{\max}(\sL_{T}^{q})]$. 
\end{lemma}
\begin{proof}
	For a fixed $q \in \RR$ with $|q|>1$ from Lemma \ref{lem:bapat-lal-pati},  
	$\lambda_{n}(\sL_{T}^q)<0$, $\lambda_{n-1}(\sL_{T}^q)>0$ and  $\det(\sL_T^q)=1-q^2<0$. 
	We recall that  $\lambda_{n-1}(\sL_T^q|v)\geq 0$. 
	Therefore, the interlacing theorem for eigenvalues of symmetric matrices gives the following. 
	$$\lambda_1(\sL_{T}^q)\geq \lambda_1(\sL_{T}^q|v) \geq  
	\cdots \geq 
	\lambda_{n-1}(\sL_{T}^q|v) > \max\left(1-q^2, \lambda_n(\sL_{T}^q)\right) \
	geq  \min \left(1-q^2, \lambda_n(\sL_{T}^q)\right).$$ 
	
	
	By using identical arguments as done in the proof of Lemma \ref{lem1:F_T_q|<1}, 
	all the non-zero roots of   $\sF_{T}^{v}(q,x)$ lie 	in the interval 
	$[\lambda_{a}(\sL_{T}^q),\lambda_{\max}(\sL_{T}^q|v)]$ for all $q\in \RR$ with $|q|>1$. 
	The proof is complete.  
\end{proof}


The following example illustrates Lemmas \ref{lem1:F_T_q|<1} and \ref{lem2:F_T_q|>1}.

\begin{example}
	\label{eg:eg_graphs}
	Let $T_1$ be the tree given in Figure   \ref{fig:min_max_eig_GTS_6}.   
	For all 
	$x\in [\lambda_{\min}(\sL_{T_1}^{q}), \lambda_{\max}(\sL_{T_1}^{q})]$ with $q=0.5$ 
	and $q=1.5$ some values of      
	$f^{\sL_{T_1}^{q}}(q,x)$, $f^{\sL_{T_1}^{q}|1}(q,x)$ and $\sF_{T_1}^{1}(q,x)$ 
	are drawn
	in Figures \ref{fig:graph_q<1} and  \ref{fig:graph_q>1} respectively. 
	$1$ is an eigenvalue of both the matrices $\sL_{T_1}^q$ and $\sL_{T_1}^q|1$. 
	Thus, $1$ is a root of $\sF_{T_1}^{1}(q,x)$. 
	Here, the solid red, dotted blue  and thick solid  black lines  are used for    
	$f^{\sL_{T_1}^{q}}(q,x)$, $f^{\sL_{T_1}^{q}|1}(q,x)$ and $\sF_T^{1}(q,x)$ respectively. 
	Arrows on lines are used for the behaviour of these polynomials 
	when $x$ decreases from   $\lambda_{\max}(\sL_{T_1}^{q})$ to $\lambda_{\min}(\sL_{T_1}^{q}).$
	These polynomials were drawn by using the computer package SageMath.
	
	\begin{figure}[h]
		\begin{minipage}[b]{0.45\linewidth}
			\centering
			\includegraphics[scale=.32]{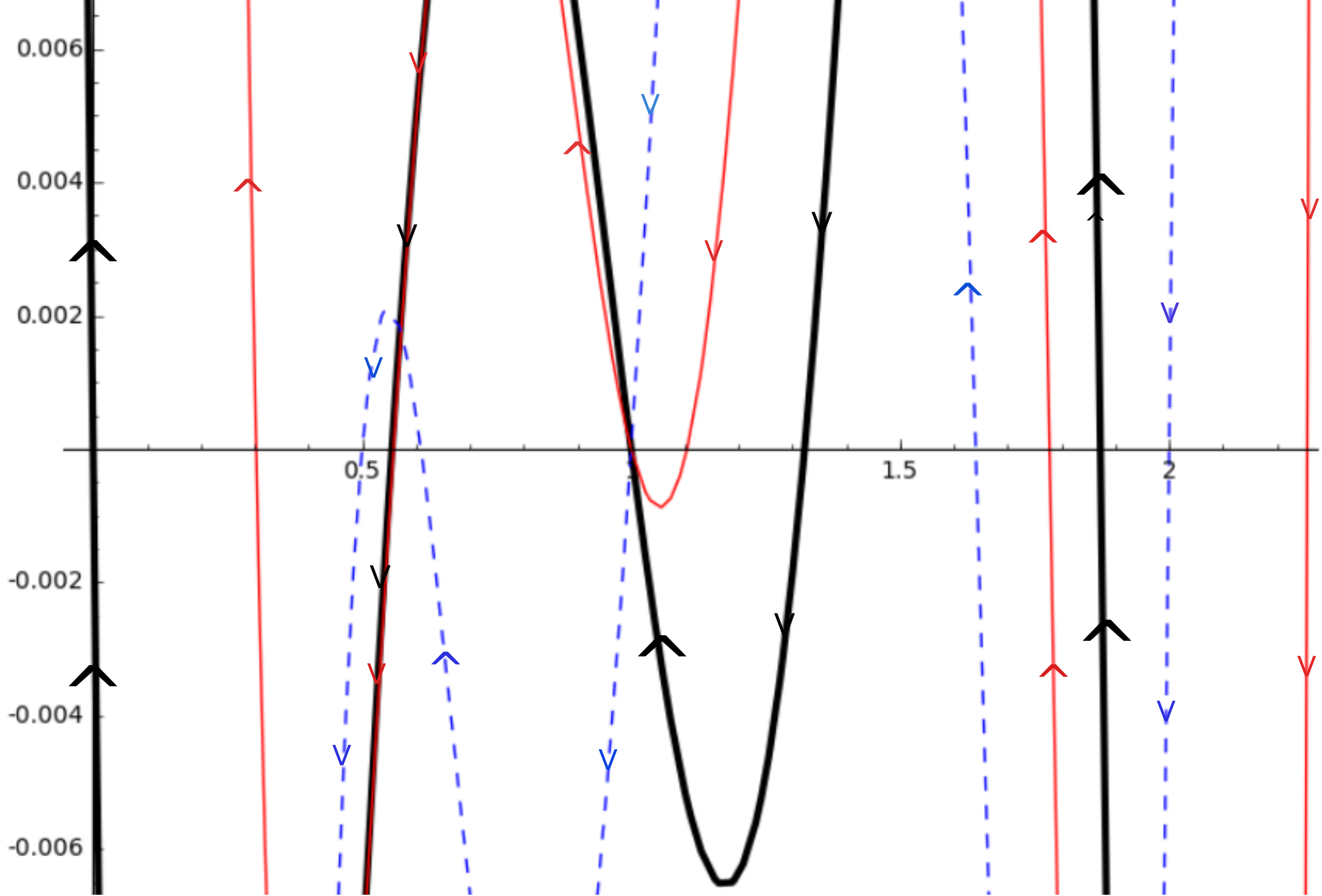}
			\caption{The values of      
				$f^{\sL_{T_1}^{q}}(0.5,x)$, $f^{\sL_{T_1}^{q}|1}(0.5,x)$ and $\sF_{T_1}^{1}(0.5,x)$.}
			\label{fig:graph_q<1}
		\end{minipage}
		\hspace{1cm}
		\begin{minipage}[b]{0.45\linewidth}
			\centering
			\includegraphics[scale=.3]{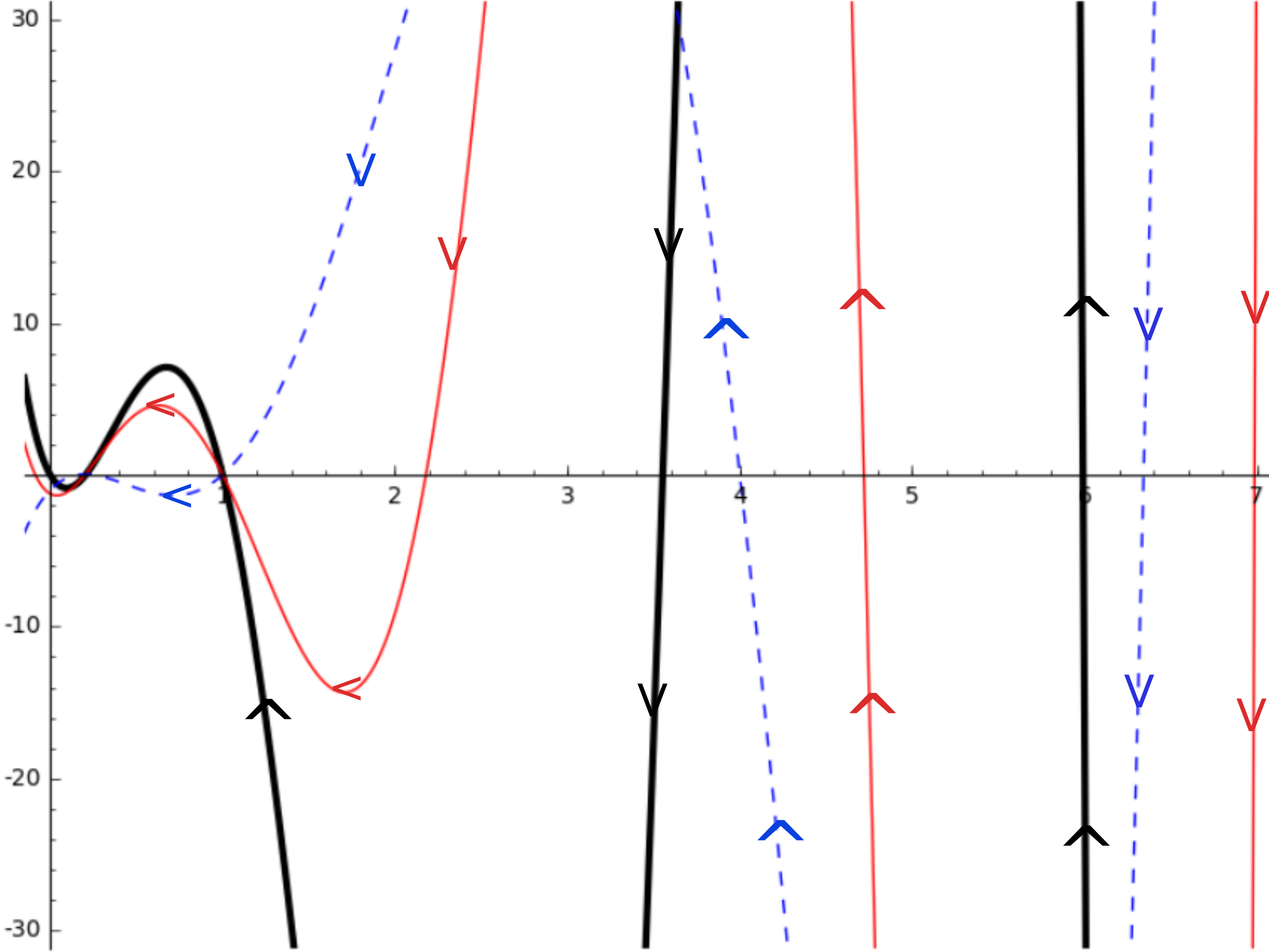}
			\caption{The values of      
				$f^{\sL_{T_1}^{q}}(1.5,x)$, $f^{\sL_{T_1}^{q}|1}(1.5,x)$ and $\sF_{T_1}^{1}(1.5,x)$.}
			\label{fig:graph_q>1}
		\end{minipage}
	\end{figure}
\end{example}

\vspace{2mm}

We recall that zero is a root of $\sF_{T}^{v}(q,x)$ with multiplicity one 
and the second smallest root of $\sF_{T}^{v}(q,x)$ lies in the interval 
$[\lambda_{a}(\sL_{T}^{q}),\lambda_{n-2}(\sL_{T}^{q}|v)]$, 
where $\lambda_{a}(\sL_{T}^{q})>0$ and 
$\lambda_{n-2}(\sL_{T}^{q}|v)>0$ are  the second smallest eigenvalues of $\sL_{T}^{q}$ 
and $\sL_{T}^{q}|v$ respectively. 
Thus, by Lemmas \ref{lem:deg_aux_poly}, \ref{lem1:F_T_q|<1} and \ref{lem2:F_T_q|>1}, we obtain the following.
\begin{remark}
	\label{rem:sgn_F_T} 
	Let $T$ be a tree on $n$ vertices and let $v\in T$. 
	Then, for a fixed $q\in \RR\backslash \{0\}$ and for  $x\in (-\infty,0)$ 
	the polynomial  $\sF_{T}^{v}(q,x)$ is positive when $n$ is even and negative  
	when $n$ is odd. Moreover, by the IVT,  for  all $x\in (0,\lambda_{a}(\sL_{T}^{q}))$, 
	we have  $\sF_{T}^{v}(q,x)< 0$ when $n$ is even and $\sF_{T}^{v}(q,x)> 0$ when $n$ is odd. 
\end{remark}

Let $T_1$ and $T_2$ be two trees with $n$ vertices such that $T_2$ covers $T_1$ in $\GTS_n$. 
For convenience, from Theorem \ref{thm:result_by_gen_lemma}, we define 
\begin{equation}
\label{eqn:def_D_T1^T2}
\sD_{T_2}^{T_1}(q,x)=f^{\sL_{T_1}^{q}}(q,x)-f^{\sL_{T_2}^{q}}(q,x)  
=  -  \frac{1}{q^2x}\sF_{P_k}^{1}(q,x)\sF_{H_1}^{1}(q,x)\sF_{H_2}^{1}(q,x), 
\mbox{ \rm{ where } } q\neq 0. 
\end{equation}

From Lemma \ref{lem:deg_aux_poly}, zero is a root of all the 
three polynomials $\sF_{P_k}^{1}(q,x)$, $\sF_{H_1}^{1}(q,x)$ 
and  $\sF_{H_2}^{1}(q,x)$ with multiplicity one. 
Thus, from \eqref{eqn:def_D_T1^T2}, zero is a root of $\sD_{T_2}^{T_1}(q,x)$ with multiplicity two. 
We need the following lemma in Section \ref{sec:eigenvalue_q_Lap}.

\begin{lemma}
	\label{lem:sgn_D_T2^T1}
	Let $T_1$ and $T_2$ be two trees with $n$ vertices such that $T_2$ covers $T_1$ in $\GTS_n$. 
	Let $\sD_{T_2}^{T_1}(q,x)$ be the polynomial defined in \eqref{eqn:def_D_T1^T2}. 
	Then, for $i=1,2$ and for a fixed $q\in \RR\backslash \{0\}$ 
	when  $x\in (-\infty,\lambda_{a}(\sL_{T_i}^{q}))-\{0\}$, we have 
	$\sD_{T_2}^{T_1}(q,x)>0$ if $n$ is even and $\sD_{T_2}^{T_1}(q,x)<0$ if $n$ is odd. 
\end{lemma} 
\begin{proof}
	By Corollary \ref{cor:ev_n-1_subtrees} and Remark \ref{rem:sgn_F_T}, 
	for a fixed $q\in \RR\backslash \{0\}$ and 
	when $x \in (0,\lambda_{a}(\sL_{T_i}^{q}))$,  
	each polynomial from $\sF_{P_k}^1(q,x)$, $\sF_{H_1}^1(q,x)$,  
	and $\sF_{H_2}^1(q,x)$ evaluates to a  negative quantity 
	if  the number of vertices of  $P_k$, $H_1$ and $H_2$ are  even respectively. 
	Similarly, each polynomial from $\sF_{P_k}^1(q,x)$, $\sF_{H_1}^1(q,x)$,  
	and $\sF_{H_2}^1(q,x)$ evaluates to a  positive quantity 
	if the number of vertices of $P_k$, $H_1$ and $H_2$ are odd respectively.  
	
	When $n$ is an even number, then from  Remark \ref{rem:even_odd_no_vertices}, 
	for all $x \in (0,\lambda_{a}(\sL_{T_i}^{q}))$, 
	either all the three polynomials  $\sF_{P_k}^1(q,x)$, $\sF_{H_1}^1(q,x)$,  
	and $\sF_{H_2}^1(q,x)$ evaluate to  negative quantities 
	or exactly one polynomial evaluates to a negative quantity 
	and other two evaluate to positive quantities. 
	Thus, by  using \eqref{eqn:def_D_T1^T2}, we get 
	$\sD_{T_2}^{T_1}(q,x)>0$ for all $x \in (0,\lambda_{a}(\sL_{T_i}^{q}))$. 
	Similarly, when $n$ is an odd number, then by Remark \ref{rem:even_odd_no_vertices} 
	and \eqref{eqn:def_D_T1^T2}, 
	$\sD_{T_2}^{T_1}(q,x)<0$ for all $x \in (0,\lambda_{a}(\sL_{T_i}^{q}))$. 
	By similar arguments, it is simple to see that for all 
	$x \in (-\infty,0)$, $\sD_{T_2}^{T_1}(q,x)>0$ 
	when $n$ is even and  $\sD_{T_2}^{T_1}(q,x)<0$ when $n$ is odd. 
	The proof is complete.
\end{proof}

\vspace{2mm}

The following lemma is an easy consequence of 
Lemmas \ref{lem1:F_T_q|<1} and \ref{lem2:F_T_q|>1}  and 
Corollary  \ref{cor:subtree_ev}.

\begin{lemma}
	\label{lem:max_ev_with_D_T1T2}
	Let $T_1$ and $T_2$ be two trees with $n$ vertices such that $T_2$ covers $T_1$ in $\GTS_n$. 
	Then for a fixed $q\in \RR\backslash \{0\}$, 	$\sD_{T_2}^{T_1}(q,x) >  0$ 
	for all $x > \lambda_{\max}(\sL_{T_i}^{q})>0,$ where  $i=1,2.$
\end{lemma} 
\begin{proof}
	We recall that $P_k$, $H_1$ and $H_2$ are the subtrees of $T_1$ and $T_2$. 
	From Corollary \ref{cor:subtree_ev}, we  see that  
	$ \max \left(\lambda_{\max}(\sL_{P_k}^{q}),\lambda_{\max}(\sL_{H_1}^{q}),
	\lambda_{\max}(\sL_{H_2}^{q}) \right) \leq \min \left(\lambda_{\max}(\sL_{T_1}^{q}),
	\lambda_{\max}(\sL_{T_2}^{q})\right).$
	Thus, from Lemmas \ref{lem1:F_T_q|<1} and \ref{lem2:F_T_q|>1},  
	the polynomials $\sF_{P_k}^{1}(q,x)$, $\sF_{H_1}^{1}(q,x)$ and $\sF_{H_2}^{1}(q,x)$ are negative 
	for all $x > \lambda_{\max}(\sL_{T_i}^{q})> 0 $ for $i=1,2.$  
	Thus, using \eqref{eqn:def_D_T1^T2}, completes the proof. 
\end{proof}

\section{Proof of Theorem \ref{thm:main_result}}
\label{sec:eigenvalue_q_Lap}

We give the proof of Theorem \ref{thm:main_result} in three subsections one for each eigenvalue. 
It is sufficient to prove the result for each pair of covering trees on $\GTS_n$.   
The following remark is straight forward from the definition of 
$f^{\sL_{T}^q}(q,x)$. 

\begin{remark}
\label{rem:sgn_f^sL_T^i}
For a fixed    $q\in \RR$ and for    $x\in (-\infty,\lambda_{\min}(\sL_{T}^{q}))$, the polynomial   
$f^{\sL_{T}^q}(q,x)$ evaluates to a positive quantity 
when $n$ is even and negative quantity when $n$ is  odd.
Moreover, when  $x\in [\lambda_{\min}(\sL_{T}^{q}),\lambda_{a}(\sL_{T}^{q})]$, 
we have  $f^{\sL_{T}^q}(q,x) \leq 0$ if $n$ is even and $f^{\sL_{T}^q}(q,x)\geq 0$ if $n$ is odd. 
We also see that $f^{\sL_{T}^q}(q,x)> 0$ for all $x> \lambda_{\max}(\sL_{T}^{q})$.
\end{remark}  

\subsection{$\lambda_{\max}(\sL_T^q)$}
\label{subsec:largest_eg}

The following result is  about monotonicity of the largest eigenvalue of $\sL_T^q$ of a tree $T$ on $\GTS_n$.

\begin{theorem}
\label{thm:ev_max_main}
Let $T_1$ and $T_2$ be two trees with $n$ vertices such that $T_2$ covers $T_1$  in $\GTS_n$. 
Then, for all $q\in \RR\backslash \{0\}$ , we have   
$\lambda_{\max}(\sL_{T_1}^{q})\leq \lambda_{\max}(\sL_{T_2}^{q}).$ 
In particular, for any tree $T$ with $n$ vertices, we have  
$\lambda_{\max}(\sL_{P_n}^q) \leq \lambda_{\max}(\sL_{T}^q) \leq \lambda_{\max}(\sL_{S_n}^q).$ 
\end{theorem}
\begin{proof}    
Assume to the contrary that   $\lambda_{\max}(\sL_{T_1}^{q})> \lambda_{\max}(\sL_{T_2}^{q})$.  
When $x=\lambda_{\max}(\sL_{T_1}^{q})$, 
by using Remark \ref{rem:sgn_f^sL_T^i} and \eqref{eqn:def_D_T1^T2}, we get 
$\sD_{T_2}^{T_1}(q,\lambda_{\max}(\sL_{T_1}^{q}))= -f^{\sL_{T_2}^{q}}(q,\lambda_{\max}(\sL_{T_1}^{q}))< 0$. 
This contradicts  Lemma \ref{lem:max_ev_with_D_T1T2}. 
Thus, $\lambda_{\max}(\sL_{T_1}^{q})\leq  \lambda_{\max}(\sL_{T_2}^{q})$ for all $q\in \RR\backslash \{0\}$.  
Using  Lemma \ref{lem:csikvari_min_max} completes the proof.
\end{proof}


In Example \ref{eg:ev_S_n}, we determine all the eigenvalues of $\sL_{S_n}^q$ for all $q\in \RR$. 
Therefore, by Theorem \ref{thm:ev_max_main}, 
we obtain an upper bound on the largest eigenvalue of $\sL_{T}^{q}$ for all $q\in \RR$. 

\begin{example}
\label{eg:ev_S_n}
Let $S_n$ be the star tree on the vertex set $[n]$ with $q$-Laplacian $\sL_{S_n}^q$. 
Without loss of generality in $S_n$, we can  assume that the vertex $1$ has degree $n-1$. 
We see that 
the only permutations $(1,j)\in \SSS_n$  contribute to  $\det(xI-\sL_{S_n}^q)$, 
where $j=1,2,\ldots,n$. 
The identity permutation contributes  $(x-1)^{n-1}(x-1-(n-2)q^2)$ and 
each of the remaining $n-1$ transpositions contribute  $-q^2(x-1)^{n-2}$. 
Therefore, 
\begin{eqnarray*}
f^{\sL_{S_n}^q}(q,x) & = & \left[(x-1)^{n-1}(x-1-(n-2)q^2)\right]
-\left[(n-1)q^2(x-1)^{n-2}\right] \\
& = & (x-1)^{n-2}\left[x^2-(2+(n-2)q^2)x-q^2+1\right].
\end{eqnarray*}
Thus, the eigenvalues of $\sL_{S_n}^q$ are  $1$ with multiplicity $n-2$ and 
$\frac{2+(n-2)q^2\pm  \sqrt{n^2q^4+4(n-1)(1-q^2)q^2}}{2}.$
\end{example}

The following corollary is an immediate consequence of Theorem \ref{thm:ev_max_main}.
\begin{corollary}
\label{cor:max_ev_bound_L_Tq}
Let $T$ be a tree on $n>1$ vertices with $q$-Laplacian $\sL_T^q$.  
Then, for all $q\in \RR$
$$\lambda_{\max}(\sL_T^q) \leq \frac{2+(n-2)q^2+ \sqrt{n^2q^4+4(n-1)(1-q^2)q^2}}{2}.$$ 
\end{corollary}

\subsection{ $\lambda_{\min}(\sL_T^{q})$}
\label{subsec:smallest_eg}
Now we prove a part of Theorem \ref{thm:main_result} about the smallest eigenvalue of $\sL_T^q$ 
when $|q|\leq 1$. 
The proof of the following theorem is very similar to the proof of Theorem \ref{thm:ev_max_main}.
 
\begin{theorem}
\label{thm:min_ev_q<1}
Let $T_1$ and $T_2$ be two trees with $n$ vertices such that $T_2$ covers $T_1$  in $\GTS_n$. 
Then, for all $q\in \RR\backslash \{0\}$ with $|q|\leq 1$,  we have  
$\lambda_{\min}(\sL_{T_1}^{q})\geq \lambda_{\min}(\sL_{T_2}^{q}).$ 
\end{theorem}
\begin{proof}
Assume to the contrary that $\lambda_{\min}(\sL_{T_1}^q) <\lambda_{\min}(\sL_{T_2}^q)$. 
It is easy to see that in this case $q\neq \pm 1$. 
Therefore from Lemma \ref{lem:bapat-lal-pati},  $0< \lambda_{\min}(\sL_{T_1}^q) <\lambda_{\min}(\sL_{T_2}^q)$. 
By Remark \ref{rem:sgn_f^sL_T^i}, 
we see that $-f^{\sL_{T_2}^q}(q,\lambda_{\min}(\sL_{T_1}^q))$ is negative  
if $n$ is even and positive if $n$ is odd.  
Therefore by \eqref{eqn:def_D_T1^T2},   
$\sD_{T_2}^{T_1}(q,\lambda_{\min}(\sL_{T_1}^q)) 
= -f^{\sL_{T_2}^q}(q,\lambda_{\min}(\sL_{T_1}^q))$ is negative 
if $n$ is even and positive if $n$ is odd which contradicts Lemma \ref{lem:sgn_D_T2^T1}.
Thus, $ \lambda_{\min}(\sL_{T_1}^q) \geq 
\lambda_{\min}(\sL_{T_2}^q) \geq 0$ and the proof is complete.
\end{proof}


We next show that going up on $\GTS_n$ decreases the smallest eigenvalue of $\sL_T^q$ 
when $q\in \RR$ with $|q|>1$.
The proof of the following theorem is very similar to the proof of Theorem \ref{thm:min_ev_q<1}. 
\begin{theorem}
\label{thm:min_ev_q>1}
Let $T_1$ and $T_2$ be two trees on  $n$ vertices such that $T_2$ covers $T_1$ in $\GTS_n$. 
Then, for all $q\in \RR$ with $|q|>1$,   we have 
$\lambda_{\min}(\sL_{T_1}^{q})\geq \lambda_{\min}(\sL_{T_2}^{q}).$ 
\end{theorem}
\begin{proof}
When $q\in \RR$ with $|q|>1$, from Lemma \ref{lem:bapat-lal-pati}, 
$\lambda_{\min}(\sL_{T_1}^{q})<0$ and $\lambda_{\min}(\sL_{T_2}^{q})<0.$
Assume to the contrary that $ \lambda_{\min}(\sL_{T_1}^q) < \lambda_{\min}(\sL_{T_2}^q)<0$.  
By  Remark \ref{rem:sgn_f^sL_T^i}, 
$f^{\sL_{T_2}^q}(q,\lambda_{\min}(\sL_{T_1}^q))$ is positive  
when $n$ is even and negative when $n$ is odd. 
Therefore, by \eqref{eqn:def_D_T1^T2},    
$\sD_{T_2}^{T_1}(q,\lambda_{\min}(\sL_{T_1}^q)) = 
-f^{\sL_{T_2}^q}(q,\lambda_{\min}(\sL_{T_1}^q))$ 
is negative when $n$ is even and positive when $n$ is odd. 
This contradicts  Lemma \ref{lem:sgn_D_T2^T1}.  
Thus, $0 > \lambda_{\min}(\sL_{T_1}^q) \geq 
\lambda_{\min}(\sL_{T_2}^q)$, completing the proof.
\end{proof}


\comment{
By combining Lemma \ref{lem:csikvari_min_max} with  
Theorems \ref{thm:min_ev_q<1} and  \ref{thm:min_ev_q>1}, 
we obtain the following  corollary. 

\begin{corollary}
\label{cor:min_ev}
Let $T_1$ and $T_2$ be two trees with $n$ vertices 
such that $T_2$ covers $T_1$  in $\GTS_n$. 
Then, for all $q\in \RR$, we have    
$\lambda_{\min}(\sL_{T_1}^{q})\geq \lambda_{\min}(\sL_{T_2}^{q}).$ 
In particular, for any tree $T$ with $n$ vertices 
$\lambda_{\min}(\sL_{P_n}^q) \geq \lambda_{\min}(\sL_{T}^q) \geq \lambda_{\min}(\sL_{S_n}^q).$ 
\end{corollary}
}

\subsection{ $\lambda_{a}(\sL_T^{q})$}
\label{subsec:2nd_min_ev}

We next show that as we go up on $\GTS_n$, the second smallest eigenvalue 
of $\sL_T^q$ increases for all $q\in \RR$.   We begin with the following.  

\begin{remark}
\label{rem:2nd_ev_L_T2} 
Let $T_1$ and $T_2$ be two trees on $n$ vertices such that  $T_2$ covers $T_1$  in $\GTS_n$. 
By the interlacing theorem,  
the smallest eigenvalue  of $\sL_{T_i}^q|\{1,k\}$ lies between 
$\lambda_{\min}(\sL_{T_i}^q)$ and $\lambda_{n-2}(\sL_{T_i}^q)$, where $i=1,2$.
\end{remark}

From Lemma \ref{lem:mult_small_ev}, we recall that  for all $q\in \RR\backslash\{0\}$, 
the algebraic multiplicity of  $\lambda_{\min}(\sL_{T}^{q})$ as an eigenvalue of $\sL_{T}^{q}$ is 1. 
This is required to prove  the following result. 

\begin{lemma}
\label{lem:min_T1<2nd_min_T2}
Let $T_1$ and $T_2$ be two trees on $n$ vertices such that $T_2$ covers $T_1$ in $\GTS_n$. 
Then, for all $q\in \RR\backslash \{0\}$, we have $\lambda_{\min}(\sL_{T_1}^q) \leq  \lambda_{a}(\sL_{T_2}^q)$.
\end{lemma}

\begin{proof}
For a fixed $q\in \RR$ with $|q|\geq 1$, From Lemma \ref{lem:bapat-lal-pati},  
we get $\lambda_{\min}(\sL_{T_1}^q) \leq 0 <  \lambda_{a}(\sL_{T_2}^q)$. 
When $q\in \RR\backslash \{0\}$ with $ |q|<1$, assume to the contrary that   
$\lambda_{\min}(\sL_{T_1}^q) > \lambda_{a}(\sL_{T_2}^q)$. 
Therefore by Lemma \ref{lem:mult_small_ev},  $\lambda_{a}(\sL_{T_1}^q) > \lambda_{\min}(\sL_{T_1}^q) > 
\lambda_{a}(\sL_{T_2}^q) > \lambda_{\min}(\sL_{T_2}^q)>0$.
From \eqref{eqn:def_D_T1^T2},  
$\sD_{T_2}^{T_1}(q,\lambda_{\min}(\sL_{T_1}^{q}))=-f^{\sL_{T_2}^q}(q,\lambda_{\min}(\sL_{T_1}^q))$. 
Thus, by  Lemma \ref{lem:sgn_D_T2^T1}, 
$f^{\sL_{T_2}^q}(q,\lambda_{\min}(\sL_{T_1}^q))$ is negative if $n$ is even and 
positive  if $n$ is odd. On the other hand by Remark    \ref{rem:sgn_f^sL_T^i} and by the IVT, either    $f^{\sL_{T_2}^q}(q,x)$ evaluates to a positive quantity  if $n$ is even 
and negative quantity if $n$ is odd for some $x \in (\lambda_{a}(\sL_{T_2}^q),\lambda_{\min}(\sL_{T_1}^q))$ 
or there is an eigenvalue $\lambda_{n-2}(\sL_{T_2}^q)$ of $\sL_{T_2}^q$ in the interval $\left[\lambda_{a}(\sL_{T_2}^q), \lambda_{\min}(\sL_{T_1}^q)\right)$.
Thus, in both the cases,  there must be an eigenvalue 
$\lambda_{n-2}(\sL_{T_2}^q)$ of $\sL_{T_2}^q$ in the interval 
$\left[\lambda_{a}(\sL_{T_2}^q), \lambda_{\min}(\sL_{T_1}^q)\right)$ such that 
$$\lambda_{a}(\sL_{T_1}^q)  >  \lambda_{\min}(\sL_{T_1}^q) > 
 \lambda_{n-2}(\sL_{T_2}^q)  \geq \lambda_{a}(\sL_{T_2}^q)> \lambda_{\min}(\sL_{T_2}^q)$$
which contradicts Remark \ref{rem:2nd_ev_L_T2} and 
hence the proof is complete. 
\end{proof}

\vspace{2mm}

The proof of the following theorem is very similar to  the proof of Theorem \ref{thm:ev_max_main}.

\begin{theorem}
\label{thm:2nd_ev_main}
Let $T_1$ and $T_2$ be two trees with $n$ vertices such that $T_2$ covers $T_1$ in $\GTS_n$. 
Then, for all $q\in \RR\backslash\{0\}$, we have   
$\lambda_{a}(\sL_{T_1}^{q})\leq \lambda_{a}(\sL_{T_2}^{q}).$ 
In particular, for any tree $T$ with $n$ vertices,   
$\lambda_{a}(\sL_{P_n}^q) \leq \lambda_{a}(\sL_{T}^q) \leq \lambda_{a}(\sL_{S_n}^q).$ 
\end{theorem}

\begin{proof}  
Assume to the contrary that  $\lambda_{a}(\sL_{T_1}^{q})> \lambda_{a}(\sL_{T_2}^{q})>0$.  
From Lemma \ref{lem:min_T1<2nd_min_T2}, $\lambda_{a}(\sL_{T_1}^{q}) > 
\lambda_{a}(\sL_{T_2}^{q}) \geq  \max \left(0,\lambda_{\min}(\sL_{T_1}^{q})\right)$.
By Remark \ref{rem:sgn_f^sL_T^i}, 
$f^{\sL_{T_1}^q}(q,\lambda_{a}(\sL_{T_2}^{q}))$
is non-positive when $n$ is even and non-negative when $n$ is  odd. 
Thus, from \eqref{eqn:def_D_T1^T2}, 
$\sD_{T_2}^{T_1}(q,\lambda_{a}(\sL_{T_2}^{q}))
=f^{\sL_{T_1}^q}(q,\lambda_{a}(\sL_{T_2}^{q}))$
is non-positive when $n$ is even and non-negative when $n$ is  odd. 
This contradicts Lemma \ref{lem:sgn_D_T2^T1}.   
By Lemma \ref{lem:csikvari_min_max}, the proof is complete.
\end{proof}

\vspace{2mm}

From Example \ref{eg:ev_S_n}, when $n>2$ we get   $\lambda_{a}(\sL_{S_n}^q)=1$. 
Thus, we obtain the following corollary.

\begin{corollary}
\label{cor:2nd_smal_ev_bound}
Let $T$ be a tree on $n>2$ vertices with $q$-Laplacian $\sL_T^q$.  
Then, for all $q\in \RR$, we have  $\lambda_{a}(\sL_T^q)\leq 1.$
\end{corollary}

The following example illustrates Theorem \ref{thm:main_result}.

\begin{example}
Let $T_1$ and $T_2$ be the two trees given in Figure \ref{fig:min_max_eig_GTS_6}. 
Clearly $T_2$ covers $T_1$ in $\GTS_6$. 
Let $\sL_{T_1}^q$ and $\sL_{T_2}^q$ be the $q$-Laplacians of $T_1$ and $T_2$ respectively. 
When $q \in \{0.1, 0.5, 1.0, 1.5,10\}$, 
the largest,  
the smallest 
and the second smallest eigenvalues 
of $\sL_{T_i}^q$  are given 
in Table  \ref{tab:min_max_eig_GTS_6},  where $i=1$, $2$. 
Calculations were done by using the computer package SageMath.
\begin{figure}[h]
\centering
\includegraphics[scale=0.8]{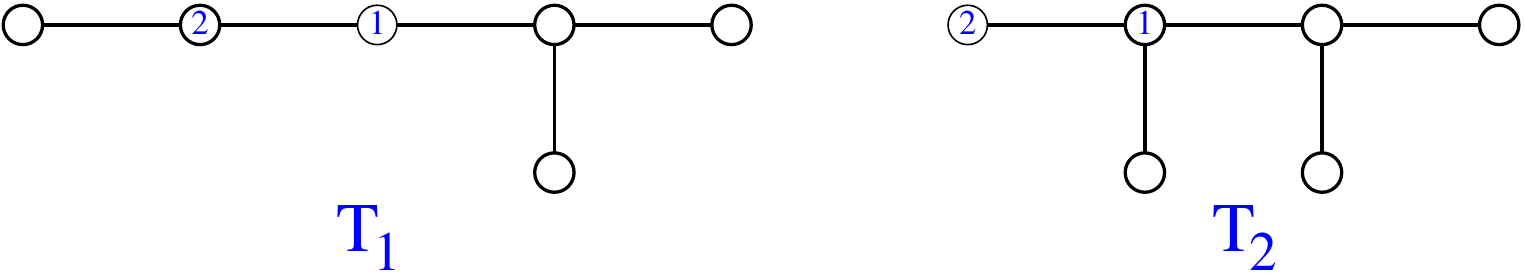}
\caption{An example of $\GTS_6$}
\label{fig:min_max_eig_GTS_6}
\end{figure} 
\begin{table}[h]
$$  \begin{array}{|r||r|r||r|r||r|r|}
\hline
& \multicolumn{2}{c||}{\lambda_{\max}(\sL_{T_i}^q)} &
\multicolumn{2}{c||}{\lambda_{\min}(\sL_{T_i}^q)} &
\multicolumn{2}{c|}{\lambda_{a}(\sL_{T_i}^q)} \\
\hline
q & \lambda_{\max}(\sL_{T_1}^q) & 
\lambda_{\max}(\sL_{T_2}^q) & 
\lambda_{\min}(\sL_{T_1}^q) & \lambda_{\min}(\sL_{T_2}^q) &
\lambda_{a}(\sL_{T_1}^q) & \lambda_{a}(\sL_{T_2}^q) \\
    \hline
    \hline
0.1 & 1.2017  & 1.2136  & 0.8208  & 0.8130  & 0.8890 &  0.9064  \\
 \hline
0.5 & 2.2566  & 2.3660  & 0.3032  & 0.2929  & 0.5586  & 0.6340   \\    
 \hline
1.0 & 4.2143  & 4.5616  &  0.0000 & 0.0000  & 0.3249  & 0.4384   \\
 \hline
1.5 & 6.9857  & 7.6742  & -0.0864  & -0.0981  & 0.2014   & 0.3258   \\
 \hline
10.0 & 202.9803  & 211.9481  & -0.0069  &  -0.0469 & 0.0070  & 0.0519   \\
 \hline
  \end{array}$$
\caption{The eigenvalues $\lambda_{max}(\sL_{T_i}^q)$, 
$\lambda_{min}(\sL_{T_i}^q)$ and  
$\lambda_{a}(\sL_{T_i}^q)$.}
\label{tab:min_max_eig_GTS_6}  
\end{table}
\end{example}

\section{Eigenvalues of the $q,t$-Laplacian}
\label{sec:ev_L^qt}

Let $T$ be a tree with $n$ vertices. 
A generalization $\sL_T^{q,t}$ was defined by Bapat and Sivasubramanian 
in \cite{bapat-siva-prod_dist_matrix}.  
Orient the edges of $T$ arbitrarily.  Since $T$ is now directed, we use
directed graph terminology for its arcs $(i,j)$.   
Define $\sL_{T}^{q,t} = (\ell_{i,j})_{1 \leq i,j \leq n}$ by $\ell_{i,j} = -q$, if 
$\{i,j\}$ is an edge $e$ in $T$ and the orientation of $e$ gives the 
arc $(i,j)$.  If the orientation of $e=\{i,j\}$ gives the arc $(i,j)$, then, 
define $\ell_{j,i} = -t$. If $\{i,j\}=e$ with $i\neq j$ is not an edge in $T$, 
then, define $\ell_{i,j}=0$. Define $\ell_{i,i} = 1 + qt(d_i -1)$, 
where $d_i$ is the degree of the $i$-th vertex in $T$.
It is easy to see that when $q=t$, $\sL_{T}^{q,t} = \sL_T^q$. 
Thus, $\sL_{T}^{q,t}$ is a generalization of $\sL_{T}^{q}$.  

Let $T$ be a tree on $n$ vertices with $q,t$-Laplacian $\sL_{T}^{q,t}$. 
For fixed but arbitrary $q,t\in \CC$, 
let $f^{\sL_{T}^{q,t}}(q,t,x)$ be the characteristic polynomial of $\sL_{T}^{q,t}$ in three variables $q$, $t$ and $x$. 
Clearly 
Lemma \ref{lem:gen_lemma} and Theorem \ref{thm:result_by_gen_lemma} go through for  
$f^{\sL_{T}^{q,t}}(q,t,x)$ when $q\neq 0$ and $t \neq 0$. 
In this case, it is easy to see that $y_1=0$, $y_2=1$, $y_3=-(x-1+qt)$ and $y_4=-1/qtx$.

We need the Interlacing Theorem for eigenvalues of Hermitian  matrices to 
generalize Lemmas \ref{lem1:F_T_q|<1} and \ref{lem2:F_T_q|>1}  to $\sL_{T}^{q,t}$. 
Hence, we require $\sL_{T}^{q,t}$ to be Hermitian. 
We see that the matrix 
$\sL_{T}^{q,t}$ is Hermitian for all $q,t\in \CC$ with $\obr{q}=t$. 
When $\obr{q}=t$, the proof of the $q,t$-version of the general lemma is 
identical to that of Lemma \ref{lem:gen_lemma}.  
Bapat and Sivasubramanian \cite{bapat-siva-prod_dist_matrix}  proved that  
$\det(\sL_{T}^{q,t})=1-qt.$ For $v\in T$, 
Nagar and Sivasubramanian \cite{mukesh-siva-immanantal_polynomial}  proved that  
$f^{\sL_{T}^{q,t}|v}(q,t,0)=(-1)^{n-1}$ for all $q,t \in \CC$. 
When $ \obr{q}=t$ in $\CC\backslash\{0\}$, we get  $qt=|q|^2>0$ and thus, 
Lemmas \ref{lem1:F_T_q|<1} and \ref{lem2:F_T_q|>1} go through for $\sL_{T}^{q,t}$.

When $\obr{q}= t \in \CC$, Theorem  \ref{thm:main_result} also goes through for 
the bivariate Laplacian matrix $\sL_{T}^{q,t}$.  
Let $\lambda_{\min}(\sL_{T}^{q,t})$, 
$\lambda_{a}(\sL_{T}^{q,t})$ and $\lambda_{\max}(\sL_{T}^{q,t})$ be the smallest, 
the second smallest and the largest eigenvalues of 
 $\sL_{T}^{q,t}$ 
 respectively.  
We get the following result as a generalization of Theorem \ref{thm:main_result}. 
As its proof is very similar to the proofs of    
Theorems \ref{thm:ev_max_main}, \ref{thm:min_ev_q<1}, \ref{thm:min_ev_q>1} and \ref{thm:2nd_ev_main}, 
we omit  and  merely state the result.
\begin{theorem}
\label{thm:q_t-Laplacian_result}
Let $T_1$ and $T_2$ be two trees on $n$ vertices such that $T_2$ covers $T_1$  in $\GTS_n$. 
Let $\sL_{T_1}^{q,t}$ and $\sL_{T_2}^{q,t}$ 
be the bivariate Laplacians of $T_1$ and $T_2$ respectively. 
Then, for  $q,t\in \CC$ with $\obr{q}=t$, we have 
$\lambda_{\min}(\sL_{T_1}^{q,t})\geq \lambda_{\min}(\sL_{T_2}^{q,t})$,  
$\lambda_{\max}(\sL_{T_1}^{q,t})\leq \lambda_{\max}(\sL_{T_2}^{q,t})$ and 
$\lambda_{a}(\sL_{T_1}^{q,t})\leq \lambda_{a}(\sL_{T_2}^{q,t}).$  
\end{theorem}

Thus, for all $\obr{q}=t\in \CC$, three eigenvalues of $\sL_{T}^{q,t}$ exhibit  monotonicity 
when we go up on $\GTS_n$. 
Moreover, for the largest and second smallest eigenvalues  max-min pair is $(S_n,P_n)$ 
while for the smallest eigenvalue max-min pair is $(P_n,S_n)$. 
It is easy to determine the eigenvalues of $\sL_{S_n}^{q,t}$ 
as  done in Example \ref{eg:ev_S_n}.
Thus, we get the following corollary of Theorem  \ref{thm:q_t-Laplacian_result}.

\begin{corollary}
Let $T$ be a tree on $n>2$ vertices with  $q,t$-Laplacian $\sL_T^{q,t}$. 
Then, for all $q,t\in \CC$ with $\obr{q}=t$, we have $\lambda_{a}(\sL_{T_1}^{q,t})\leq 1$ and 
$$\lambda_{\max}(\sL_T^{q,t}) \leq \frac{2+(n-2)qt+ \sqrt{n^2q^2t^2+4(n-1)(1-qt)qt}}{2}.$$ 
\end{corollary} 

When $q=1/t\in \CC$ with $t\neq 0$ and $\obr{q}\neq t $, 
the matrix $\sL_{T}^{q,t}$ is no longer Hermitian 
but our result follows from   
Nagar and Sivasubramanian \cite[Remark 34]{mukesh-siva-immanantal_polynomial}. 
There it was proved that when $q=1/t\in \CC$ with $t\neq 0$, $\det(xI-\sL_{T}^{q,t})=\det(xI-L_T)$,  
where $L_T$ is the Laplacian matrix of $T$. 
Thus, we have  $\lambda_{\min}(\sL_{T}^{q,t})=\lambda_{\min} (L_{T})$,  
$\lambda_{\max}(\sL_{T}^{q,t})=\lambda_{\max} (L_{T})$ and 
$\lambda_{a}(\sL_{T}^{q,t})=\lambda_{a} (L_{T})$.
Moreover, we have $\lambda_{\max}(\sL_{T}^{q,t})=\lambda_{\max} (L_{T}) \leq n$ 
when  $q=1/t\in \CC$ with $t \neq 0$. 

One special case of $\sL_{T}^{q,t}$ is 
obtained when we set $q=\imath$ and $t=-\imath$, where $\imath = \sqrt{-1}$. 
In this case, the matrix $\sL_{T}^{q,t}$ is the Hermitian 
positive semidefinite 
(see Bapat and Sivasubramanian \cite{bapat-siva-prod_dist_matrix}).
The bivariate Laplacian matrix $\sL_{T}^{q,t}$
is called the Hermitian Laplacian matrix of $T$ when $q=\imath$ and 
$t=-\imath$ defined by Yu and Qu \cite{yu-qu-hermitian-Laplacian}. 
In this case, we have both $\obr{q}=t$ and $q=1/t$.

\section{Eigenvalue monotonicity of $\ED_T$} 
\label{sec:expon_dist_mat}

Let $T$ be a tree on $n$ vertices with exponential distance matrix $\ED_T$. 
Let the eigenvalues of $\ED_T$ be  
$\lambda_{\max}(\ED_T)=\lambda_1(\ED_T)\geq \lambda_2(\ED_T) \geq 
\cdots \geq \lambda_n(\ED_T)=\lambda_{\min}(\ED_T)$. 
When $q=\pm 1$, it is simple to see that the eigenvalues of  $\ED_T$ are $n$ and 
$0$ with multiplicities $1$ and $n-1$ respectively. 
Thus, in this case, the eigenvalues of $\ED_T$ are constant on $\GTS_n$. 
By Lemmas \ref{lem:bapat-lal-pati} and \ref{lem:bapat_ED_ev}, 
we see that $\ED_T$ is a positive definite matrix when $q\in \RR$ with $|q|<1$ 
and $\ED_T$ has exactly one positive eigenvalue when $q\in \RR$ with $|q|>1$.
Let $\lambda_{\min}(\ED_T)$, $\lambda_{\max}(\ED_T)$ and 
$\lambda_{2}(\ED_T)$  be the smallest, the  largest and 
the second largest eigenvalues of $\ED_T$.  
The following theorem is our main result of this section.
\begin{theorem}
\label{thm:main_result_E}
Let $T_1$ and $T_2$ be two trees with $n$ vertices such that $T_2$ covers $T_1$ in $\GTS_n$. 
\begin{enumerate}
\item If $q\in \RR$ with $|q|<1$, then,   $\lambda_{\min}(\ED_{T_1})\geq \lambda_{\min}(\ED_{T_2})$, 
$\lambda_{2}(\ED_{T_1})\geq \lambda_{2}(\ED_{T_2})$ and 
$\lambda_{\max}(\ED_{T_1})\leq \lambda_{\max}(\ED_{T_2})
.$ 
\item If $q\in \RR$ with $|q|>1$, then,    $\lambda_{\min}(\ED_{T_1})\leq \lambda_{\min}(\ED_{T_2})$, 
$\lambda_{2}(\ED_{T_1})\leq \lambda_{2}(\ED_{T_2})$ and  
$\lambda_{\max}(\ED_{T_1})\geq \lambda_{\max}(\ED_{T_2})
.$ 
\end{enumerate}
\end{theorem}

\begin{proof} 
 When $q\in \RR$ with $|q|<1$, by using Lemma \ref{lem:bapat_ED_ev}, 
all the eigenvalues of $\ED_{T_1}$ and $\ED_{T_2}$ are positive and they are 
\begin{eqnarray}
\lambda_{\max}(\ED_{T_1})=\frac{1-q^2}{\lambda_{\min}(\sL_{T_1}^{q})} & \geq & 
\frac{1-q^2}{\lambda_{a}(\sL_{T_1}^{q})} 
\geq  \cdots 
\geq
\frac{1-q^2}{\lambda_{1}(\sL_{T_1}^{q})}
= \frac{1-q^2}{\lambda_{\max}(\sL_{T_1}^{q})}  \nonumber \\ 
 \mbox{ and }  \ \ \
\lambda_{\max}(\ED_{T_2})=\frac{1-q^2}{\lambda_{\min}(\sL_{T_2}^{q})} & \geq &
\frac{1-q^2}{\lambda_{a}(\sL_{T_2}^{q})} 
\geq  \cdots 
\geq
\frac{1-q^2}{\lambda_{1}(\sL_{T_2}^{q})}
= \frac{1-q^2}{\lambda_{\max}(\sL_{T_2}^{q})}
\label{eqn:ev_ED_T_q<1}
\end{eqnarray}
respectively. Similarly, when $|q|>1$, by Lemma \ref{lem:bapat_ED_ev}, 
both the matrices $\ED_{T_1}$ and $\ED_{T_2}$ have exactly one positive eigenvalue and  their eigenvalues are 
\begin{eqnarray}
\lambda_{\max}(\ED_{T_1})=\frac{1-q^2}{\lambda_{\min}(\sL_{T_1}^{q})} & \geq & 
\frac{1-q^2}{\lambda_{\max}(\sL_{T_1}^{q})} 
\geq  \cdots 
\geq
\frac{1-q^2}{\lambda_{n-1}(\sL_{T_1}^{q})}
= \frac{1-q^2}{\lambda_{a}(\sL_{T_1}^{q})} \nonumber \\
 \mbox{ and } \ \ \ 
\lambda_{\max}(\ED_{T_2})=\frac{1-q^2}{\lambda_{\min}(\sL_{T_2}^{q})} & \geq &
\frac{1-q^2}{\lambda_{\max}(\sL_{T_2}^{q})} 
\geq \cdots
\geq
\frac{1-q^2}{\lambda_{n-1}(\sL_{T_2}^{q})}
= \frac{1-q^2}{\lambda_{a}(\sL_{T_2}^{q})}
\label{eqn:ev_ED_T_q>1}
\end{eqnarray}
respectively. Thus, by using  \eqref{eqn:ev_ED_T_q<1}, \eqref{eqn:ev_ED_T_q>1} and 
Theorem \ref{thm:main_result}  the proof is complete.
\end{proof} 


Thus, for all $q\in \RR$, three eigenvalues of $\ED_T$ exhibit  monotonicity 
when we go up on $\GTS_n$ and max-min pair is either $(P_n,S_n)$ or $(S_n,P_n)$. 
 
\comment{
By using  Lemma \ref{lem:csikvari_min_max}, 
we get the following simple corollary of Theorem \ref{thm:main_result_E}. 
We omit the proof.
\begin{corollary}
\label{cor:2nd_ev_ED}
Let $T$ be any tree on $n$ vertices and let $q\in \RR$. 
\begin{enumerate}
\item  If $|q|<1$, then,  
$\lambda_{\min}(\ED_{P_n})\geq \lambda_{\min}(\ED_T) \geq \lambda_{\min}(\ED_{S_n})$, 
$\lambda_{a}(\ED_{P_n})\geq \lambda_{a}(\ED_T) \geq \lambda_{a}(\ED_{S_n})$ and  
$\lambda_{\max}(\ED_{P_n})\leq \lambda_{\max}(\ED_T) \leq \lambda_{\max}(\ED_{S_n}).$  
\item If $|q|>1$, then,  $\lambda_{\min}(\ED_{P_n})\leq \lambda_{\min}(\ED_T) \leq 
\lambda_{\min}(\ED_{S_n})$, $\lambda_{a}(\ED_{P_n})\leq \lambda_{a}(\ED_T) 
\leq \lambda_{a}(\ED_{S_n})$   and  $\lambda_{a}(\ED_{P_n})\geq \lambda_{a}(\ED_T) 
\geq \lambda_{a}(\ED_{S_n}).$ 
\end{enumerate}
\end{corollary}
}

\subsection{$q,t$-exponential distance matrix}

We consider the bivariate exponential distance matrix $\ED_T^{q,t}$ of a tree $T$.
Orient the tree $T$ as done in Section \ref{sec:ev_L^qt}.  Thus each 
directed arc $e$ of $E(T)$ has a unique reverse arc $e_{rev}$ and
we assign a variable weight $w(e) = q$ and $w(e_{rev}) = t$ or 
vice versa.  If the path
$P_{i,j}$ from vertex $i$ to vertex $j$ has the sequence of 
edges $P_{i,j} = (e_1,e_2, \ldots, e_p)$,  assign it weight
$w_{i,j} = \prod_{e_k \in P_{i,j}} w(e_k)$.  
Define $w_{i,i}=1$ for $i=1,2,\ldots,n$. 
Define the bivariate $q,t$-exponential distance matrix 
$\ED_T^{q,t} = (w_{i,j})_{1 \leq i,j \leq n}$. 
Clearly, when $q=t$, we will have $\ED_T^{q,t} = \ED_T$.
Bapat and Sivasubramanian in 
\cite[Theorem 3.2]{bapat-siva-prod_dist_matrix} 
showed the following bivariate counterpart of Lemma \ref{lem:bapat_ED_ev}.

\begin{lemma}[Bapat and Sivasubramanian]
\label{lem:bapat_E_biv_inverse}
Let $T$ be a tree on $n$ vertices with $q,t$-Laplacian  $\sL_{T}^{q,t}$ 
and $q,t$-exponential distance matrix $\ED_T^{q,t}$. 
When $qt \neq 1$, then   
$(\ED_T^{q,t})^{-1}=\frac{1}{1-qt} \sL_{T}^{q,t}.$
Moreover,  $\frac{1-qt}{\lambda_i(\sL_{T}^{q,t})}$ is an eigenvalue of $\ED_T^{q,t}$, 
where $\lambda_i(\sL_{T}^{q,t})$ is an eigenvalue of $\sL_{T}^{q,t}$.
\end{lemma}

From Theorem \ref{thm:q_t-Laplacian_result}, 
it is easy to see that Theorem \ref{thm:main_result_E}  goes through
for the bivariate $q,t$-exponential distance matrix $\ED_T^{q,t}$
when $q,t \in \CC$ with $\obr{q}=t$ and $qt \not= 1$.

\section*{Acknowledgement}
The author acknowledges support from DST, New Delhi for providing a Senior Research Fellowship. 
Our main theorem in this work was in its conjecture form, 
tested using the computer package ``SageMath''. 
We thank the authors for generously releasing SageMath as an open-source package.

\bibliographystyle{acm}
\bibliography{main}
\end{document}